\documentclass[10pt,a4paper,reqno]{amsart}
\usepackage{acronym}
\usepackage{amsfonts}
\usepackage{amsmath}
\usepackage{amssymb}
\usepackage{amsthm}
\usepackage{bm}
\usepackage{braket}
\usepackage{cite}
\usepackage{color}
\usepackage{geometry}
\usepackage{graphicx}
\usepackage{hyperref}
\usepackage{mathrsfs}
\usepackage{mathtools}
\usepackage{setspace}
\usepackage{stmaryrd}
\usepackage{subcaption}
\usepackage{tikz}

\usepackage{epsfig}
\usepackage{setspace}
\usepackage{booktabs}
\usepackage{threeparttable}
\usepackage{diagbox}

\newgeometry{top=2cm,bottom=2cm,outer=1.5cm,inner=1.5cm}
\newcommand{\bse}{\begin{subequations}}
\newcommand{\ese}{\end{subequations}}

\newtheorem{theorem}{Theorem}

\newtheorem{lemma}[theorem]{Lemma}
\newtheorem{assume}[theorem]{Assume}

\newtheorem{corollary}[theorem]{Corollary}
\newtheorem{proposition}[theorem]{Proposition}
\newtheorem{rhp}[theorem]{Riemann-Hilbert Problem}
\numberwithin{equation}{section}

\title[Long-time Asymptotic Behavior of the coupled dispersive AB system in Low Regularity Spaces ]{Long-time Asymptotic Behavior of the coupled dispersive AB system in Low Regularity Spaces}

\author{JinYan Zhu}
\address[JY]{School of Mathematical Sciences, Shanghai Key Laboratory of Pure Mathematics and Mathematical Practice\\
East China Normal University \\ Shanghai 200241 \\ People's Republic of China}
\author{Yong Chen$^*$}
\address[YC]{School of Mathematical Sciences, Shanghai Key Laboratory of Pure Mathematics and Mathematical Practice \\
East China Normal University \\ Shanghai 200241 \\ People's Republic of China}
\address[YC]{College of Mathematics and Systems Science \\ Shandong University of Science and Technology \\ Qingdao 266590 \\ People's Republic of China}
\email{ychen@sei.ecnu.edu.cn}

\begin{document}

\begin{abstract}
In this paper, we mainly investigate the long-time asymptotic behavior of the solution for the coupled dispersive AB system with weighted Sobolev initial data, which allows soliton solutions via the Dbar steepest descent method.
 Based on the spectral analysis of Lax pair, the Cauchy problem of the coupled dispersive AB system is transformed into a Riemann-Hilbert problem, and its existence and uniqueness of the solution is proved by the vanishing lemma. The stationary phase points play an important role in the long-time asymptotic behavior. We demonstrate that in any fixed time cone $\mathcal{C}\left(x_{1}, x_{2}, v_{1}, v_{2}\right)=\left\{(x, t) \in \mathbb{R}^{2} \mid x=x_{0}+v t, x_{0} \in\left[x_{1}, x_{2}\right], v \in\left[v_{1}, v_{2}\right]\right\}$, the long-time asymptotic behavior of the solution for the coupled dispersive AB system can be expressed by $N(\mathcal{I})$ solitons on the discrete spectrum, the leading order term $\mathcal{O}(t^{-1 / 2})$ on the continuous spectrum and the allowable residual $\mathcal{O}(t^{-3 / 4})$.
\end{abstract}

\maketitle

\section{Introduction}
In the past few decades, many solutions of integrable equations have been studied, including soliton solutions, breather solutions, rogue wave solutions and so on\cite{VAE-1978-ZETF,KN-1978-JMP,VAK-1989-PD,LB-2020-ND,ZL-2013-PRE}. These solutions are obtained by classical inverse scattering method, Riemann-Hilbert method under the condition of no reflection\cite{JH-2018-JNMP,YB-2019-NAR,LT-2022-NAR}. However, for continuous spectrum or reflection, it is necessary to analyze the asymptotic behavior of the solution. In fact, as early as 1973, Manakov et al. first studied the long-time asymptotic behavior of the solution of the fast decay initial value problem for nonlinear integrable systems by using the inverse scattering method\cite{M-1974-SPJ,AN-1973-JMP}. In 1976, Zakharov and Manakov gave the first term asymptotic expression that the solution of the initial value problem for nonlinear Schr\"{o}dinger (NLS) equation which depends explicitly on the initial value \cite{ZM-1976-SPJ}. In 1981, Its transformed the long-time behavior of the solution for the initial value problem of the NLS equation into the local Riemann-Hilbert problem(RHP) in the small neighborhood of the steady-state phase point by using the single value deformation theory, which provided a set of feasible and strict ways to analyze the long-term behavior of the integrable equation\cite{Its-1981-DAN}. In 1993, Deift and Zhou developed the nonlinear descent method for solving oscillatory RHP based on the classical descent method, and studied the asymptotic properties of the solution for mKdV equation with decaying initial value \cite{DZ-1993-AM}. Since then, more and more scholars have paid attention to the nonlinear descent method to study the long-time asymptotic behavior of the solution for the initial value problem of integrable systems, and many equations have been studied by this method \cite{XD-1994-UT,KG-2009-MPAG,AD-2013-N,PJ-1999-CPDE,LJ-2018-JDE,JE-2015-JDE,XFC-2013-MPAG,AA-2009-JMA}.
Later, McLaughlin and Miller extended the classical Deift-Zhou steepest descent method to Dbar steepest descent method, which was successfully used to study the asymptotic stability of NLS multiple soliton solutions \cite{MM-2006-IMRN,MM-2008-IMRN,DM-2019-FIC}, and the long-time proximity of KdV equation and NLS equation \cite{SR-2016-CMP,MR-2018-AN,RJ-2018-CMP,LP-2018-AI,P-2017-N}. TheDbar steepest descent method is to transform the discontinuous part on the jump line into the form of Dbar problem, rather than the analysis of the asymptotic properties of orthogonal polynomials in singular integrals on the jump line.
Then it became a more powerful tool in the study of long-time asymptotic for integrable equations, such as NLS equation \cite{DM-2019-N}, the Hirota equation \cite{YTL-2021-ar}, the Fokas-Lenells equation \cite{CF-2022-JDE}, the modified Camassa-Holm equation \cite{YF-2022-AM}, the short-pulse equation \cite{YF-2021-JDE} and so on.

The coupled dispersion AB system \cite{JD-1979-PRSL}
\begin{equation}\label{AB}
\begin{array}{l}
A_{x t}-\alpha A-\beta A B=0, \\
B_{x}+\frac{\gamma}{2}\left(|A|^{2}\right)_{t}=0,
\end{array}
\end{equation}
 describes the evolution of micro stable or unstable wave packets in baroclinic shear flow through a quasi geostrophic two-layer model on a beta-plane, where $A\equiv A(x,t)$ is the complex function representing the amplitude of the data packet, the real function $B\equiv B(x,t)$ is the measure of the average flow change caused by the baroclinic wave packet, and $\alpha$ represents the critical situation of shear and the real parameter $\beta,\gamma$ represent the nonlinear coefficient. And system (\ref{AB}) meets the following compatibility condition:
 $$
\frac{\gamma}{\beta}\left|A_{t}\right|^{2}+B^{2}+\frac{2 \alpha}{\beta} B=f(t),
$$
where $f(t)$ is an integral function of time.
 If $A(x, t)=1 / \sqrt{\beta \gamma} \psi_{x}(x, t)$ and $B(x,t)=\pm[\cos (\psi(x, t))-\alpha] / \beta \text { with } \beta \gamma>0$,  system (\ref{AB}) can reduce to $\psi_{x t}=\pm \sin \psi.$
If  $A(x, t)=1 / \sqrt{-\beta \gamma} \psi_{x}(x, t)$ and $B(x,t)=\pm[\cos (\psi(x, t))-\alpha] / \beta \text { with } \beta \gamma<0$,  system (\ref{AB}) can reduce to $\psi_{x t}=\pm \sinh \psi.$

 In recent years, many studies have been done for system (\ref{AB}). For example,  the soliton solutions of system (\ref{AB}) are obtained by the inverse scattering method \cite{JD-1979-PS}.  Kamchatnov and Pavlov found the periodic wave solution of system (\ref{AB})\cite{AM-1995-JPA}. Guo et al. have studied its classical Darboux transformation(DT) and N-flod DT and obtained breather solutions and multi-soliton solutions \cite{GR-2013-ND,GL-2015-AMC}. The rogue wave solution of system (\ref{AB}) was also obtained by generalized DT under the condition $\alpha=0, \beta=\gamma=1,f(t)=1$\cite{WX-2015-CN}.  In addition, the high-order strange wave and modulation instability of system (\ref{AB}) were further given via generalized DT \cite{WY-2015-CHAOS}. Recently, Yan et al. studied the long-time asymptotic behavior of solutions whose initial values belong to Schwarz space through the classical Deift-Zhou descent method \cite{SZ-2021-JMAA}.

 In this paper, we mainly study the long-time asymptotic behavior of the solution of system (1) satisfying the following initial value problem
 $$
A{(x,0)}=A_0(x)\in H^{1,1}(\mathbb{R}),
B{(x,0)}=B_0(x)\in H^{1,1}(\mathbb{R}),
$$
where $H^{1,1}(\mathbb{R})$ is the weighted Sobelev space, its expression is
$$
H^{1,1}(\mathbb{R})=\left\{f(x) \in L^{2}(\mathbb{R}): f^{\prime}(x), x f(x) \in L^{2}(\mathbb{R})\right\}.
$$
The main tool used here is the Dbar steepest descent method. Compared with Deift-Zhou steepest descent method, Dbar method avoids estimating Cauchy integral operator in $L^P$ space and rewrites the discontinuous part of RHP into Dbar problem which can be solved by integral equation. Compared with previous studies \cite{SZ-2021-JMAA}, we consider the initial value in a wider space instead of in Schwarz space and allow the existence of soliton solutions.

The context of this paper is as follows:

In section 2, we analyze the spectrum of the system (\ref{AB}) based on Lax pair. The analyticity, symmetry and asymptotic behavior of characteristic function and scattering matrix are studied. In terms of asymptotic behavior, two singular points $k=0$ and $k=\infty$ need to be considered. We also prove that when the initial value belongs to weighted Sobelev space, $r(k)$ belongs to $H^{1,1}(\mathbb{R})$.

In section 3, the RHP of system (\ref{AB}) is constructed by the piecewise smooth property of matrix $M(x,t)$.

In section 4, we prove the existence and uniqueness of the solution of RHP by using vanishing lemma.

In section 5,  by introducing $T(k)$ function, we change $M(x,k)$ into $M^{(1)}(x,k)$ and get a new RHP about $M^{(1)}(x,k)$. This is mainly to allow the jump matrix near the phase points to have two triangular decomposition.

In section 6, the contour is deformed near the phase points, and a mixed $\bar{\partial}$-RHP is obtained by defining $\mathcal{R}^{(2)}$.

In section 7, we decompose the mixed $\bar{\partial}$-RHP into a model RHP of $M_{rhp}(k)$ and the pure $\bar{\partial}$ problem of $ M^{(3)}(k)$.

In section 8, the RHP  about $M_{rhp}(k)$ is studied and the long-time behavior of the soliton solution is analyzed. It is obtained that the soliton solution can be represented by a ring region, and the applicability near the phase points is also explained. Secondly, the error function is calculated with RHP of a small norm.

In section 9, the pure $\bar{\partial}$ problem for $M^{(3)}(k)$ is analyzed.

In section 10, using the above deformation and results, Eq.(\ref{mt}) is obtained, and then the long-time behavior of the solution of the coupled dispersion AB system (\ref{AB}) in the case of weighted Sobelev initial value is obtained, which is given in the form of theorem.

\section{The spectral analysis}
System (\ref{AB}) has Lax pair is \cite{JD-1981-PRSLA}
\begin{equation}
\begin{array}{l}
\phi_{x}=\hat{U} \phi=\left(\begin{array}{cc}
-i k & \frac{\sqrt{\beta \gamma}}{2}A \\
-\frac{\sqrt{\beta \gamma}}{2} {A}^{*} & i k
\end{array}\right) \phi, \\
 \phi_{t}=\hat{V} \phi=\left(\begin{array}{cc}
\frac{i(\alpha+\beta B)}{4 k} & -\frac{i \sqrt{\beta \gamma}}{4 k} A_{t} \\
-\frac{i \sqrt{\beta \gamma}}{4 k}{A}^{*}_{t} & -\frac{i(\alpha+\beta B)}{4 k}
\end{array}\right) \phi.
\end{array}
\end{equation}
In order to better facilitate the presentation of the results, we mainly consider the case of $\beta\gamma=1$. At this time, the Lax pair of Eq.(\ref{AB}) can be simplified to
\begin{equation}
\Phi_{x}=X\Phi, ~~~\Phi_{t}=T\Phi,\label{lax}
\end{equation}
where
\begin{equation}
\begin{array}{l}
X=-ik\sigma_3+U,~~~~
T= \frac{i\alpha}{4k} \sigma_{3}+T_1,
\end{array}
\end{equation}
with
\begin{equation}
\begin{array}{l}
\sigma_{3}=\left[\begin{array}{cc}
1 & 0 \\
0 & -1
\end{array}\right], \quad U(x,t)=\frac{1}{2}\left(\begin{array}{cc}
0 & A \\
-A^{*} & 0
\end{array}\right), \quad T_1(x, t, k)=\frac{i}{4 k}\left(\begin{array}{cc}
\beta B & -A_{t} \\
-{A}^{*}_{t} & - \beta B
\end{array}\right).
\end{array}
\end{equation}
The spectral parameter $k\in \mathbb{C}$, under the rapidly decaying initial condition
$$
 A_{0}(x),B_0(x) \rightarrow 0 \quad \text {as} \quad x \rightarrow \infty,
$$
make the following correction transformation
\begin{equation}
\Phi=\Psi e^{-ikx\sigma_3+ \frac{i\alpha}{4k}t\sigma_3},\label{pp}
\end{equation}
the modified Lax pair can be obtained
\begin{equation}
\begin{aligned}\label{ut}
\Psi_{x}+i k\left[\sigma_{3}, \Psi\right] &=U \Psi, \\
\Psi_{t}-\frac{i\alpha}{4k}\left[\sigma_{3}, \Psi\right] &=T_1 \Psi,
\end{aligned}
\end{equation}
which can be written in full derivative form
$$
d(e^{i\left(k x-\frac{\alpha}{4k} t\right)\widehat{\sigma}_{3}} \Psi(x, t, k))=e^{i\left(k x-\frac{\alpha}{4k} t\right)\widehat{\sigma}_{3}}[Udx+T_1dt)\Psi].
$$
From the above Lax pair (\ref{ut}), we can see that it has two singular points $k = 0$ and $k = \infty$. Therefore, we need to investigate the different expansion forms of the characteristic function at two singular points. After simple calculation, the following form is obtained
$$
\begin{aligned}
&\Psi(x, t, k)=I+\frac{\Psi^1}{k}+\mathcal{O}\left(k^{-2}\right), \quad k \rightarrow \infty, \\
&\Psi(x, t, k)=\Psi^{0}+\Psi^{1}k+\mathcal{O}\left(k^2\right), \quad k \rightarrow 0,
\end{aligned}
$$
where
\begin{equation}\label{po}
\Psi^{0}(x,t)=\left(\begin{array}{cc}
e^{\frac{1}{2}\int_{-\infty}^{x}\frac{\beta A B}{A_t}dx} & e^{\frac{1}{2}\int_{x}^{\infty}\frac{ AA^{*}_t }{\beta B }dx} \\
 e^{\frac{1}{2}\int_{x}^{\infty}\frac{ A^{*}A_t }{\beta B }dx} & e^{\frac{1}{2}\int_{-\infty}^{x}\frac{\beta A^{*} B}{A^{*}_t}dx}
\end{array}\right),
\Psi^1_{12}=-\frac{i}{4}A, \Psi^1_{21}=-\frac{i}{4}A^{*}.
\end{equation}
The asymptotic solution $\Psi(x,t,k)$ satisfies
$$
\Psi_{\pm}(x, t, k) \rightarrow I, \quad x \rightarrow \pm \infty,
$$
the modified solution satisfies the following integral equation
\begin{equation}
\Psi_{\pm}(x, t, k)=I+\int_{\pm \infty}^{x} \mathrm{e}^{-i k(x-y) \sigma_{3}} U(y, t) \Psi_{\pm}(y, t, k) \mathrm{e}^{i k(x-y) \sigma_{3}} \mathrm{~d} y.\label{jost}
\end{equation}

Dividing $\Psi_{\pm}$ into columns as $\Psi_{\pm}=\left(\Psi_{\pm}^{(1)}, \Psi_{\pm}^{(2)}\right)$, due to the structure of the potential $U$, and Volterra integral equation (\ref{jost}), we have
\begin{proposition}
For $A(x),B(x) \in H^{1,1}(\mathbb{R})$ and $t \in \mathbb{R}^{+}$, there exist unique eigenfunctions $\Psi_{\pm}$ which satisfy Eq.(\ref{jost}), respectively, and have the following properties:

$\bullet$  $\text{det} \Psi_{\pm}(x,t,k)=1,  x, t \in \mathbb{R}$;

$\bullet$ $\left[\Psi_{-}\right]_{1},\left[\Psi_{+}\right]_{2}$ are analytic in $\mathbb{C}^{+}$ and continuous in $\mathbb{C}^{+}\cup \mathbb{R}$;

 $\bullet$ $\left[\Psi_{+}\right]_{1},\left[\Psi_{-}\right]_{2}$ are analytic in $\mathbb{C}^{-}$, and continuous in $\mathbb{C}^{-}\cup \mathbb{R}$.

\end{proposition}
 The partition of $\mathbb{C}^{\pm}$ is shown in Fig.\ref{1t}.
\begin{figure}
\center
\begin{tikzpicture}\usetikzlibrary{arrows}
\coordinate [label=0: Im $k$] ()at (0,2.8);
\coordinate [label=0: Re $k$] ()at (3,-0.2);
\coordinate [label=0:] ()at (2,0.1);
\coordinate [label=0:] ()at (-2.6,0.1);
\path [fill=pink] (-3.5,2.5)--(3.5,2.5) to
(3.5,0)--(-3.5,0);
\coordinate [label=0: $\mathbb{C}^{+}$] ()at (1,1);
\coordinate [label=0: $\mathbb{C}^{-}$] ()at (1,-1);
\draw[->, ] (3,0)--(3.5,0);
\draw[->,] (0,2)--(0,2.5);
\draw[thin](0,2)--(0,-2.3);
\draw[thin](3.5,0)--(-3.5,0);
\end{tikzpicture}
\caption{\small Definition of the $\mathbb{C}^{+}=\{k \mid \operatorname{Im} k>0\},~\mathbb{C}^{-}=\{k \mid \operatorname{Im} k<0\}$ }\label{1t}
\end{figure}
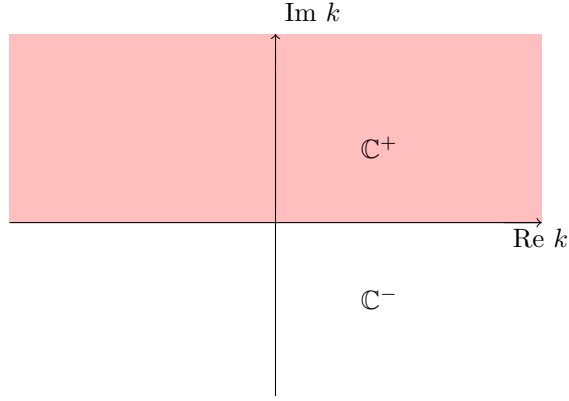
It can be seen from the equation that $\Psi_{+}e^{-ikx\sigma_3+ \frac{i\alpha}{4k}t\sigma_3}$ and $\Psi_{-} e^{-ikx\sigma_3+ \frac{i\alpha}{4k}t\sigma_3}$ are the solutions of Lax equation (\ref{jost}), so they are linearly related
\begin{equation}
\Psi_{-}e^{-ikx\sigma_3+ \frac{i\alpha}{4k}t\sigma_3} =\Psi_{+} e^{-ikx\sigma_3+ \frac{i\alpha}{4k}t\sigma_3}S(t,k), \quad k \in \mathbb{R},\label{ss}
\end{equation}
where
$$
S(t,k)=\left[\begin{array}{rr}
s_{11} & s_{12} \\
s_{21} & s_{22}
\end{array}\right].
$$
\begin{proposition}
$\Psi(x,t,k)$ and $S(t,k)$ satisfy the symmetry relation
\begin{equation}
\Psi(x,t,k)=\sigma_{0} \Psi^{*}\left(x,t,k^{*}\right)\sigma_{0}^{-1},~~~S(t,k)=\sigma_{0} S^{*}\left(t,k^{*}\right)\sigma_{0}^{-1},
\end{equation}
where
$$
\sigma_{0}=\left[\begin{array}{rr}
0 & 1 \\
-1 & 0
\end{array}\right].
$$
\end{proposition}
\begin{proof}
The symmetry of $\Phi$ can be easily obtained from Lax equation
$$
\Phi(x,t,k)=\sigma_{0} \Phi^{*}\left(x,t,k^{*}\right)\sigma_{0}^{-1},
$$
then contact transformation (\ref{pp}), which is easy to verify
$$
\Psi(x,t,k)=\sigma_{0} \Psi^{*}\left(x,t,k^{*}\right)\sigma_{0}^{-1}.
$$
Then according to the relationship (\ref{ss}), there are
$$
S(t,k)=\sigma_{0} S^{*}\left(t,k^{*}\right)\sigma_{0}^{-1}.
$$
\end{proof}
Based on the symmetry condition of $S(k)$, we can get
$$
s_{11}(k)=s_{22}^{*}(k^{*}), s_{12}(k)=-s_{21}^{*}(k^{*}).
$$
$S(t,k)$ is called scattering matrix and $s_{i,j}(i,j=1,2)$ are usually called scattering data. Moreover, Eq.(\ref{ss}) leads to
\begin{equation}\label{s112}
s_{11}(k)=\operatorname{det}\left(\left[\Psi_{-}\right]_{1},\left[\Psi_{+}\right]_{2}\right), ~s_{12}(k)=e^{2ikx-\frac{i\alpha}{2k}t} \operatorname{det}\left(\left[\Psi_{-}\right]_{2},\left[\Psi_{+}\right]_{2}\right).
\end{equation}
On the basis of  the properties of $\Psi$ in Proposition 1, we can get
\begin{proposition}
The scattering matrix $S(t,k)$ and scattering data $s_{11}(k),s_{12}(k)$ satisfies:

$\bullet$  $\text{det} S(k)=1, \quad x, t \in \mathbb{R}$;

$\bullet$ $|s_{11}(k)|^{2}+|s_{12}(k)|^{2}=1$ for $k \in \mathbb{R}$;

$\bullet$  $s_{11}(k)$ is analytic in $\mathbb{C}^{+}$ and continuous in $\mathbb{C}^{+}\cup \mathbb{R} ;$

 $\bullet$ $s_{12}(k)$ is continuous on the real $k$;

$\bullet$  $s_{11}(k) =1+\mathcal{O}(k^{-1})$ as $k \rightarrow \infty; ~~s_{12}(k) = \mathcal{O}(k^{-1})$ as $k \rightarrow \infty;$

$\bullet$  $s_{11}(k) = \varphi_0+\mathcal{O}(k^2)$ as $k \rightarrow 0; ~~s_{12}(k) =\mathcal{ O}(k^2)$ as $k \rightarrow 0, $

 where $\varphi_0=e^{\frac{1}{2}\int_{-\infty}^{\infty}\frac{\beta A^{*}B}{A^{*}_t}dx}-e^{\frac{1}{2}\int_{-\infty}^{\infty}\frac{A^{*} A_t}{\beta B}dx}.$

\end{proposition}
The reflection coefficient is described below, which is usually defined as
\begin{equation}\label{rk}
r(k)=\frac{s_{12}(k)}{s_{11}(k)},
\end{equation}
using the first two properties in proposition 3, we have
$$
1+|r(k)|^{2}=\frac{1}{|s_{11}(k)|^{2}}.
$$
In order to deal with our future work, we assume that the initial data meets the following assumptions.
\begin{assume}
The initial data $A,B\in H^{1,1}(\mathbb{R})$ and it generates generic scattering data which satisfy that:

 $\bullet$ $s_{11}(k)$ has no zeros on $\mathbb{R}$;

$\bullet$ $s_{11}(k)$ only has finite number of simple zeros.
\end{assume}

Based on the above analysis, we can get the following theorem about $r(k)$.
\begin{theorem}
For any given initial value $A,B\in H^{1,1}(\mathbb{R})$, we have $r(k) \in H^{1,1}(\mathbb{R})$.
\end{theorem}
\begin{proof}
According to the definition of $r(k)$ we have and the asymptotic behavior of $s_{11}$ and $s_{21}$ at singular points in Proposition 3, we can obtain
$$
r(k)=\mathcal{O}(k^{-1}),k\rightarrow \infty;~~~~~r(k)=\mathcal{O}(k^{2}),k\rightarrow 0.
$$
From the proposition estimation in Appendix A that
$$
s_{11}(k), s_{12}(k) \in L^{2}(\mathbb{R}),
$$
which lead to
$$
r(k) \in L^{2}(\mathbb{R}).
$$
In addition, it can also be obtained $s_{11}^{\prime}(k), s_{12}^{\prime}(k) \in L^{2}(\mathbb{R})$. Therefore
$$
r^{\prime}(k)=\frac{s_{12}^{\prime}(k) s_{11}(k)-s_{11}^{\prime}(k) s_{12}(k)}{s_{11}^{2}(k)} \in L^{2}(\mathbb{R}).
$$
Finally come to the map $A_0,B_0\rightarrow r(k)$ is Lipschitz continuous from $H^{1,1}(\mathbb{R})$ into $H^{1,1}$.
\end{proof}

\section{The construction of a RHP}
Suppose $s_{11}(k)$ has N simple zeros $\mathcal{K}=\left\{k_{j}\right\}_{j=1}^ {N} \subset \mathbb{C}^{+}$, reviewing the symmetry of $S$, we can see that there are also $N$ simple zeros $\mathcal{K}^{*}=\{k_{j}^{*}\}_{j=1}^ {N} \subset \mathbb{C}^{-}$. Then define a meromorphic function $M(x,t,k)$ as

\begin{equation}
M(x, t, k)=\left\{\begin{array}{ll}
\left(\frac{[\Psi_{-}]_{1}}{s_{11}(k)}, [\Psi_{+}]_{2}\right), & k \in \mathbb{C}^{+}, \\
\left([\Psi_{+}]_{1}, \frac{[\Psi_{-}]_{2}}{s_{22}(k)}\right), & k \in \mathbb{C}^{-},
\end{array}\right.
\end{equation}
which solves the following RHP.

\begin{rhp}\label{r1}
Find a matrix-valued function $M(x,t,k)$ which satisfies:

$\bullet$  Analyticity: $M(x,t,k)$ is meromorphic in $\mathbb{C}\backslash\mathbb{R}$ and has single poles;

$\bullet$   Symmetry: $M(x,t,k)=\sigma_{0} M^{*}\left(x,t,k^{*}\right)\sigma_{0}^{-1};$

$\bullet$ Jump condition: $M(x,t,k)$ has continuous boundary values $M_{\pm}(x,t,k)$ on $\mathbb{R}$ and
\begin{equation}
M_{+}(x, t, k)=M_{-}(x, t, k) V(k), \quad k \in \mathbb{R},
\end{equation}
where
\begin{equation}
V(k)=\left(\begin{array}{cc}
1+|r(k)|^2 & r^{*}(k^{*}) e^{-2 i t \theta(k)} \\
r(k) e^{2 i t \theta(k)} & 1
\end{array}\right);\label{vv}
\end{equation}

$\bullet$   Asymptotic behaviors:
$$
M(x, t, k)=I+\mathcal{O}\left(k^{-1}\right), \quad k \rightarrow \infty;
$$

$\bullet$  Residue conditions: $M$ has simple poles at each point in $k_j\in \mathbb{C}^{+}$ and $k_j^{*}\in \mathbb{C}^{-}$ with:
\begin{equation}
\begin{array}{l}
\operatorname{Res}_{k=k_{j}} M(z)=\lim _{k \rightarrow k_{j}} M(k)\left(\begin{array}{cc}
0 & 0 \\
c_{j} e^{2 i t \theta\left(k_{j}\right)} & 0
\end{array}\right), \\
\operatorname{Res}_{k={k}_{j}^{*}} M(k)=\lim _{k \rightarrow {k}_{j}^{*}} M(k)\left(\begin{array}{cc}
0 & -{c}_{j}^{*} e^{-2 i t \theta\left(k_j^{*}\right)} \\
0 & 0
\end{array}\right),
\end{array}
\end{equation}
where $\theta=k\frac{x}{t}-\frac{\alpha}{4k},c_{j}=\frac{s_{12}(k_j)}{s_{11}'(k_j)}$.
\end{rhp}
\underline{}The solution of the coupled dispersion AB system (\ref{AB}) can be expressed by
\begin{equation}
A= 4 i\lim _{k \rightarrow \infty} (k M)_{12}(x, t, k),~~B=-\frac{4 i}{\beta} \lim _{k \rightarrow \infty} \frac{d}{d t}(k M)_{11}.
\end{equation}

\section{Existence and uniqueness of solution of RH problem}

In this chapter, we will mainly focus on the existence and uniqueness of the solution of the RH problem constructed above. Our idea here is to use the vanishing lemma to show that the integral equation has only zero solution to the homogeneous equation, and then show that the equation has a unique solution.

In order to facilitate the application of the subsequent lemma, we focus all the residue conditions in RHP \ref{r1} on the circle $\{ \kappa_j,\kappa^{*}_j\}$  centered on the pole $\{k_j\in \mathbb{C}^{+},k^{*}_j\in\mathbb{C}^{-}\}$. As long as these circles are small enough, they will not intersect. See Fig. \ref{2t} for details.

\begin{figure}[htpb]
{\includegraphics[width=6.9cm,height=5.5cm]{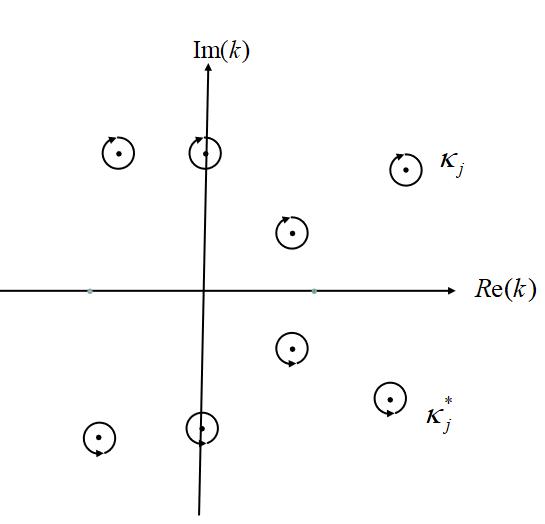}}
\centering
\caption{\small  Distribution of  $\Sigma^{(1)}$.}\label{2t}
\end{figure}
Let
$$
\Sigma^{(1)}=\mathbb{R}\cup \left\{\kappa_{j}\right\}_{j=1}^{ N} \cup\left\{\kappa^{*}_{j}\right\}_{j=1}^{ N},
$$
and rewrite RHP \ref{r1} as follows:
\begin{rhp}\label{r111}
Find a matrix-valued function $M(x,t,k)$ which satisfies:

$\bullet$  Analyticity: $M(x,t,k)$ is analytic in $\mathbb{C}\backslash\Sigma^{(1)}$;

$\bullet$   Symmetry: $M(x,t,k)=\sigma_{0} M^{*}\left(x,t,k^{*}\right)\sigma_{0}^{-1};$

$\bullet$ Jump condition: $M(x,t,k)$ has continuous boundary values $M_{\pm}(x,t,k)$ on $\Sigma^{(1)}$ and
\begin{equation}\label{v1}
M_{+}(x, t, k)=M_{-}(x, t, k) V^{\prime}(k), \quad k \in \mathbb{R},
\end{equation}
where
\begin{equation}
V^{\prime}(k)= \begin{cases}\left(\begin{array}{ll}
1+|r(k)|^{2} & r^{*}(k^{*}) e^{-2 i k x} \\
r(k) e^{2 i k x} & 1
\end{array}\right), & k \in \mathbb{R}, \\
\left(\begin{array}{cc}
1 & 0 \\
\frac{c_{j} e^{2 i k x}}{k-k_{j}} & 1
\end{array}\right), & k \in \kappa_{j}, \\
\left(\begin{array}{cc}
1 & \frac{{c}^{*}_{j} e^{-2 i k x}}{k-{k}^{*}_{j}} \\
0 & 1
\end{array}\right), & k \in \kappa^{*}_{j};\end{cases}
\end{equation}

$\bullet$  Normalization:

$$
\begin{aligned}
&(a)~~M(k,x)=I+\mathcal{O}\left(k^{-1}\right),\quad  k \rightarrow \infty,\\
&(b)~~ M(k,x)=\mathcal{O}\left(k^{-1}\right), \quad  k \rightarrow \infty.\\
\end{aligned}
$$
\end{rhp}
Here, we deliberately turn $t$ into a gray variable, mainly to better show the importance of $x$ in the proof process. In normalization (a) is mainly to reconstruct the potential $A,B$. The role of (b) is to illustrate the existence and uniqueness of solutions to the RHP of type (a).

As we all know, $V^{\prime}(k)$ has the following decomposition
$$
V^{\prime}(k)=\left(I-\Lambda^{-}\right)^{-1}\left(I+\Lambda^{+}\right).
$$
Bring it into (\ref{v1}) and define it as $\omega$
\begin{equation}\label{w1}
\omega=M_{+}\left(I+\Lambda^{+}\right)^{-1}=M_{-}\left(I-\Lambda^{-}\right)^{-1},
\end{equation}
the above formula can be rewritten as
$$
M_{+}-M_{-}=\omega(\Lambda^{+}+\Lambda^{-}).
$$
Using Plemelj formula, it is easy to give the solution of solvable RHP(\ref{r111}) equation with regularized boundary (a) in the form of
\begin{equation}
M(x, k)=I+\frac{1}{2\pi i} \int_{\Sigma^{(1)}} \frac{\omega(\Lambda^{+}+\Lambda^{-})}{s-k} d s,
\end{equation}
the solution of the equation with boundary condition (b),
\begin{equation}\label{mb}
M(x, k)=\frac{1}{2\pi i} \int_{\Sigma^{(1)}} \frac{\omega(\Lambda^{+}+\Lambda^{-})}{s-k} d s.
\end{equation}
The following main task is to prove that the solution with boundary condition (b) tends to 0, that is, Eq.(\ref{mb}) tends to 0, which is also obvious.
Write $M(x,k)$ expansion as
\begin{equation}\label{mr}
M(x, k)=\frac{1}{2\pi i} \left(\int_{\mathbb{R}} \frac{\omega(\Lambda^{+}+\Lambda^{-})}{s-k} d s+\int_{\Sigma^{(1)}\backslash \mathbb{R}} \frac{\omega(\Lambda^{+}+\Lambda^{-})}{s-k}d s\right).
\end{equation}

Here we define
$$
\mathcal{N}(k)=M(k) M({k}^{*})^{H},
$$
where $'H'$ represents conjugate transpose of the matrix, and we only need to prove
\begin{equation}\label{n1}
\int_{\mathbb{R}} \mathcal{N}_{\pm}(k) d k=0
\end{equation}
below. Obviously, according to Schwartz's reflection principle, $\mathcal{N}$ is analytical in $\mathbb{C}\backslash \Sigma^{(1)}$.

For $k\in \mathbb{R}$, we have
$$
\mathcal{N}_{+}(k)=M_{+}(k)M_{-}(k)^{H}=M_{-}(k) V^{\prime}(k) M_{-}(k)^{H},
$$
$$
\mathcal{N}_{-}(k)=M_{-}(k) M_{+}(k)^{H}=M_{-}(k) V^{\prime}(k)^{H} M_{-}(k)^{H}.
$$
 Therefore, if Eq.(\ref{n1}) is satisfied, there is
$$
\int_{\mathbb{R}} M_{-}(k)\left(V^{\prime}(k)+V^{\prime}(k)^{H}\right) M_{-}(k)^{H}=0,
$$
since $k\in \mathbb{R}$, so $V^{\prime}(k)=V^{\prime}(k)^{H}$. Then we know that $V^{\prime}(k)+V^{\prime}(k)^{H}$ is also an Hermite and is positive definite. Therefore, if the above formula holds, there must be $M_{-}(k)=0$ on $\mathbb{R}$. So
$$
M_{+}(k)=M_{-}(k) V^{\prime}(k)=0, ~~~~~k \in \mathbb{R}.
$$
According to Morera's theorem, we know that $M(k)$ is analytic in the neighborhood of each point on $\mathbb{R}$. And $M(k)$ tends to disappear.

For $k\in \Sigma^{(1)}\backslash \mathbb{R}$. then we have
$$
\begin{aligned}
\mathcal{N}_{+}(k)&=M_{+}(k) M_{-}(k^{*})^{H}=M_{-}(k) V^{\prime}(k)\left(\mathcal{N}_{-}(k^{*})^{-1}\right)^{H} M_{+}(k^{*})^{H} \\
&=M_{-}(k) M_{+}(k^{*})^{H}=\mathcal{N}_{-}(k).
\end{aligned}
$$
Similarly, it can be seen from the Morera's theorem that $\mathcal{N}(k)$ is analytic in $\Sigma^{(1)}\backslash \mathbb{R}$. And $M_{\pm}(x, \cdot) \in L^{2}(\mathbb{R})$, we know that $\mathcal{N}(k)$ is integrable. According to the integral form of Eq.(\ref{mr}),
$$
\begin{aligned}
M(x, k)&=\frac{1}{2\pi i} \int_{\Sigma^{(1)}\backslash \mathbb{R}} \frac{\omega(\Lambda^{+}+\Lambda^{-})}{s-k}d s\\
&=\frac{1}{2 k \pi i} \int_{\Sigma^{(1)}\backslash \mathbb{R}} \frac{s}{s-k} \omega(x, s)\left(\Lambda^{+}+\Lambda^{-}\right) d s-\frac{1}{2 k \pi i} \int_{\Sigma^{(1)}\backslash \mathbb{R}}\omega(x, s)\left(\Lambda^{+}+\Lambda^{-}\right) d s,
\end{aligned}
$$
the Taylor expansion of $k$ shows that $M(x, k) \sim \mathcal{O}\left(k^{-1}\right), k \in \mathbb{C} \backslash \mathbb{R}$. And then we know that $\mathcal{N}(k) \sim \mathcal{O}\left(k^{-2}\right), k \in \mathbb{C} \backslash \mathbb{R}$. Then the Eq.(\ref{n1}) can be obtained by Cauchy's integral theorem and Jordon's theorem. We can apply the same argument $k\in \mathbb{R}$ to $k\in \Sigma^{(1)}\backslash \mathbb{R}$, and we can also get $M(x, k)=0$. Therefore, it is concluded that $M(x, k)=0$ is on the whole complex plane.

So using the vanishing lemma and Fredholm alternative theorem can get the following proposition.
\begin{proposition}
Assuming that the initial date $A,B \in H^{1,1}(\mathbb{R})$, then the RHP(\ref{r111}) with normalization boundary condition (a) has a unique solution $M(x, k)$.
\end{proposition}
\section{Conjugation}
Through the form of the jump matrix and the residual condition, it can be found that the long-term asymptotic of RHP1 is affected by the growth and attenuation of the exponential function $e^{\pm2 i t \theta(k)}$. Therefore, it is necessary to deal with the jump matrix in Eq. (\ref{vv}) and the oscillation term in the residual condition, and decompose the jump matrix according to the sign change graph of $Re(i\theta)$ to ensure that any jump matrix is bounded in a given area.
We further rewrite $\theta$ in the form
$$
\theta=-\frac{\alpha k}{4k_0^2}-\frac{\alpha}{4k}=-\frac{\alpha}{4}(\frac{k}{k_0^2}+\frac{1}{k}),
$$
which leads to
\begin{equation}\label{itheta}
Re(i\theta)=\frac{Im(k)\alpha}{4}(\frac{1}{k_0^2}-\frac{1}{|k|^2}).
\end{equation}
For convenience, we mainly consider $\alpha x<0,t>0$,and can be treated similarly for $t<0$. Without losing generality, it can be assumed that $\alpha<0$ and $x>0$. Based on this, we find two phase points of function $\theta(k)$ are $\pm k_0$, where $k_0=\sqrt{-\frac{\alpha t}{4 x}}$.

According to the sign change of $Re(i\theta)$, the attenuation region of oscillation term $e^{2\pm it\theta}$ can be obtained when $t\rightarrow\infty$. which is shown in Fig. \ref{3t}.
\begin{figure}[htpb]
\centering
{\includegraphics[height=0.4\textwidth]{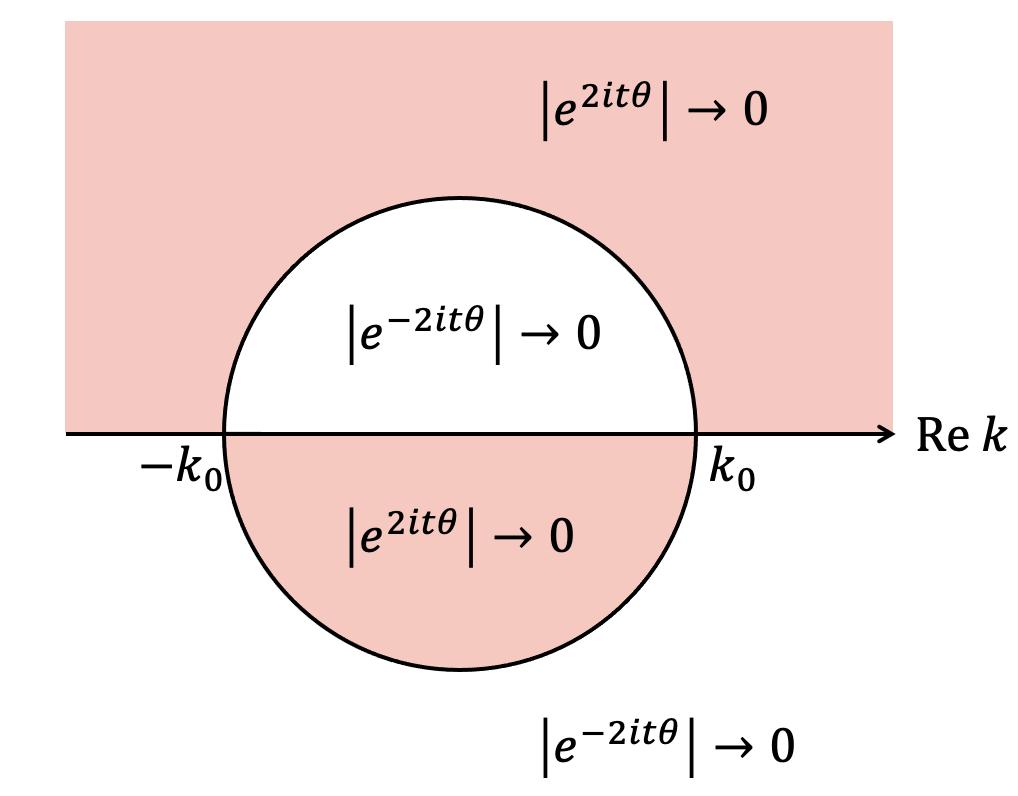}}

~~~~~~~~~~~~\caption{\small  Decay domains of exponential oscillatory terms.}\label{3t}
\end{figure}

The jump matrix $V(k)$ in RHP1 has the following decomposition
$$
V(k)=\left(\begin{array}{cc}
1 & {r}^{*}(k^{*}) e^{-2 i t \theta} \\
0 & 1
\end{array}\right)\left(\begin{array}{cc}
1 & 0 \\
r(k) e^{2 i t \theta} & 1
\end{array}\right),~~~ |k|>k_0,
$$
and
$$
V(k)=\left(\begin{array}{cc}
1 & 0 \\
\frac{r(k)}{1+|r(k)|^{2}} e^{2 i t \theta} & 1
\end{array}\right)\left(\begin{array}{cc}
1+|r(k)|^{2} & 0 \\
0 & \frac{1}{1+|r(k)|^{2}}
\end{array}\right)\left(\begin{array}{cc}
1 & \frac{{r}^{*}(k^{*})}{1+|r(k)|^{2}} e^{-2 i t \theta} \\
0 & 1
\end{array}\right), |k|<k_0.
$$
In order to remove the diagonal matrix in the second decomposition, we introduce a scalar RHP about $\delta(k)$ by using the method in Ref. \cite{CF-2022-JDE}.
\begin{rhp}\label{brh}
Find a scalar function $\delta(k)$ which satisfies:

$\bullet$  $\delta(k)$ is analytic in  $\mathbb{C}\backslash\mathbb{R}$;

$\bullet$ $
\delta_{+}(k)=\delta_{-}(k)\left(1+|r(k)|^{2}\right), \quad  |k|<k_0;$

~~$\delta_{+}(k)=\delta_{-}(k), \quad  |k|>k_0 $;

$\bullet$  $\delta(k) \rightarrow 1, \quad \text { as } k \rightarrow \infty.$
\end{rhp}
Using the Plemelij formula, it is easy to write the unique solution of the above RHP \ref{brh} as
\begin{equation}
\delta(k)=\exp \left[i \int_{-k_0}^{k_0} \frac{\nu(s)}{s-k} d s\right],
\end{equation}
where $\nu(s)=-\frac{1}{2 \pi} \log \left(1+|r(s)|^{2}\right)$.
The trace formula can be given directly according to the calculation
\begin{proposition}
\begin{equation}
s_{11}(k)=\prod_{j=1}^{N} \frac{k-k_{j}}{k-{k}^{*}_{j}} \exp \left[-i \int_{-\infty}^{+\infty} \frac{\nu(z)}{z-k} d z\right],
\end{equation}
\begin{equation}
s_{22}(k)=\prod_{j=1}^{N} \frac{k-{k}^{*}_{j}}{k-k_{j}} \exp \left[i \int_{-\infty}^{+\infty} \frac{\nu(z)}{z-k} d z\right].
\end{equation}
\end{proposition}
\begin{proof}
From the above, we know that $s_{11}(k)$ and $s_{22}(k)$ are analytic on $\mathbb{C}^{+}$ and $\mathbb{C}^{-}$  respectively, and according to the assumption, the discrete zeros of $s_{11}(k)$ and $s_{12}(k)$ are $k_j$ and $k^{*}_j$ respectively, so let
\begin{equation}
\beta^{+}(k)=s_{11}(k) \prod_{j=1}^{N} \frac{k-{k}^{*}_{j}}{k-k_{j}}, \quad \beta^{-}(k)=s_{22}(k) \prod_{j=1}^{N} \frac{k-k_{j}}{k-{k}^{*}_{j}},\label{s22}
\end{equation}
they are still analytic on $\mathbb{C}^{+}$, $\mathbb{C}^{-}$, but they no longer have zeros. And have $\beta^{+}(k) \beta^{-}(k)=s_{11}(k) s_{22}(k), k \in \mathbb{R}.$ It can be obtained that
$$
\beta^{+}(k) \beta^{-}(k)=s_{11} s_{22}=\frac{1}{1+|r(k)|^{2}}, \quad k \in \mathbb{R}
$$
through
$$
\operatorname{det} S(k)=s_{11}(k) s_{22}(k)-s_{21}(k) s_{12}(k)=1,
$$
which can leads to that
$$
\log \beta^{+}(k)-\left(-\log \beta^{-}(k)\right)=-\log \left(1+|r(k)|^{2}\right), \quad k \in \mathbb{R}
$$
using Plemelj formula, we can get
\begin{equation}
\beta^{\pm}(k)=\exp \left(\pm i \int_{-\infty}^{\infty} \frac{\nu(z)}{z-k} d z\right), \quad k \in \mathbb{C}^{\pm}.\label{s11}
\end{equation}
Therefore, the trace formula can be obtained by bringing Eq.(\ref{s11}) into (\ref{s22}).
\end{proof}

Next, for the convenience of follow-up research, we first introduce some notations
\begin{equation}
\begin{array}{l}
\Delta_{k_{0}}^{+}=\left\{j \in\{1, \ldots, N\} \mid |k_j|>k_0 \right\}, \\
\Delta_{k_{0}}^{-}=\left\{j \in\{1, \ldots, N\} \mid |k_j|<k_0\right\}.\\
\end{array}
\end{equation}
In addition, the function is introduced
\begin{equation}
T(k)=\prod_{j \in \Delta_{k_0^{-}}} \frac{k-{k}^{*}_{j}}{k-k_{j}} \delta(k).\label{tk}
\end{equation}

\begin{proposition}
The function $T(k)$ defined by Eq.(\ref{tk}) has the following properties:

$\bullet$ $T(k)$ is meromorphic function in $\mathbb{C}\backslash(-k_0,k_0)$. For $j\in\Delta_{k_{0}}^{-}$, it is a simple pole at $k_j$ and a simple zero at $k_j^{*}$. Other places are non-zero and analytic;

$\bullet$ $T^{*}(k^{*})T(k)=1$, for $\mathbb{C}\backslash(-k_0,k_0)$;

$\bullet$ For $k\in(-k_0,k_0)$, the boundary values $T_{\pm}(k)$ satisfy
$$
T_{+}(k)=T_{-}(k)\left(1+|r(k)|^{2}\right), k \in\left(-k_0,k_0\right);
$$

$\bullet$ When $|k|\rightarrow\infty$, $|\arg (k)| \leq c<\pi$
$$
T(k)=1+\frac{i}{k}\left[2 \sum_{j \in \Delta_{k_0}^{-}} \operatorname{Im} k_{j}-\int_{-k_0}^{k_0} \nu(s) d s\right]+O\left(k^{-2}\right);
$$

 $\bullet$ $k \rightarrow\pm k_0,$ along any ray $\pm k_0+e^{i\phi}\mathbb{R}_{+}$ with $|\phi| \leq c<\pi$
 $$
\left|T(k)-T_{0}\left(\pm k_0\right)\left(k\mp k_{0}\right)^{i \nu\left(\pm k_0\right)}\right| \leq c\left|k\mp k_{0}\right|^{\frac{1}{2} }
$$
where
$$
\begin{aligned}
&T_{0}\left(\pm k_0\right)=\prod_{j \in \Delta_{k_{0}}^{-}}\left(\frac{\pm k_0-{k}^{*}_{j}}{\pm k_0-k_{j}}\right) e^{i \beta\left(k, \pm k_0\right)},\\
&\beta(k,-k_0)=\left(i \int_{-k_0}^{-k_0+1} \frac{\nu\left(-k_0\right)}{z-k} d z+i \int_{-k_0}^{k_0} \frac{\nu(z)-\chi_1(z) \nu\left(-k_0\right)}{z-k} d z\right),\\
&\beta(k,k_0)=\left(i \int_{k_0-1}^{k_0} \frac{\nu\left(k_0\right)}{z-k} d z+i \int_{-k_0}^{k_0} \frac{\nu(z)-\chi_2(z) \nu\left(k_0\right)}{z-k} d z\right),\\
&\chi_{1}(k)=\left\{\begin{array}{lc}
1, & -k_0<k<-k_0+1, \\
0, & \text { elsewhere },
\end{array} \quad \chi_{2}(k)=\left\{\begin{array}{lc}
1, & k_0-1<k<k_0 ,\\
0, & \text { elsewhere }.
\end{array}\right.\right.
\end{aligned}
$$
\end{proposition}
Through the defined $T(k)$, the following transformation can be done to eliminate the diagonal matrix in the decomposition of the jump matrix in the interval $[-k_0,k_0]$.
\begin{equation}
M^{(1)}(k)=M(k) T(k)^{-\sigma_{3}},
\end{equation}
which satisfies the following RHP:
\begin{rhp}\label{4r}
Find a matrix $M^{(1)}(k)$ that satisfies:

$\bullet$ $M^{(1)}$ is analytic within $\mathbb{C}\backslash \mathbb{R}\cup\mathcal{K}\cup\mathcal{K}^{*};$

$\bullet$  $M^{(1)}(k)=\sigma_{0} (M^{(1)})^{*}\left(k^{*}\right)\sigma_{0}^{-1};$

$\bullet$$M^{(1)}(k)=I+\mathcal{O}\left(k^{-1}\right), \quad k \rightarrow \infty;$

$\bullet$ $M^{(1)}_{+}( k)=M^{(1)}_{-}( k) V^{(1)}(k), \quad k \in \mathbb{R},$

where
$$
V^{(1)}(k)=\left\{\begin{array}{c}
\left(\begin{array}{cc}
1 & 0 \\
\frac{r}{1+|r|^{2}} T_{-}^{-2} e^{2 i t \theta(k)} & 1
\end{array}\right)\left(\begin{array}{cc}
1 & \frac{{r}^{*}}{1+|r|^{2}} T_{+}^{2} e^{-2 i t \theta(k)} \\
0 & 1
\end{array}\right), \quad |k|<k_0, \\
\left(\begin{array}{cc}
1 & r^{*} T^{2} e^{-2 i t \theta(k)} \\
0 & 1
\end{array}\right)\left(\begin{array}{cc}
1 & 0 \\
r T^{-2} e^{2 i t \theta(k)} & 1
\end{array}\right), \quad |k|>k_0;
\end{array}\right.
$$

$\bullet$  $M^{(1)}(k)$ has simple poles at each point in $k_j\in \mathbb{C}^{+}$ and $k_j^{*}\in \mathbb{C}^{-}$ with:
\begin{equation}
\begin{array}{l}
\underset{k={k}_{j}}{\operatorname{Res}} M^{(1)}(k)=\left\{\begin{array}{ll}
\underset{k \rightarrow k_{j}}\lim  M^{(1)}(k)\left(\begin{array}{cc}
0 & c_{j}^{-1}\left(\frac{1}{T}\right)^{\prime}\left(k_{j}\right)^{-2} e^{-2 i t \theta\left(k_{j}\right)} \\
0 & 0
\end{array}\right), \quad j \in \Delta_{k_{0}}^{-}, \\
\underset{k \rightarrow k_{j}}\lim  M^{(1)}(k)\left(\begin{array}{cc}
0 & 0 \\
c_{j} T\left(k_{j}\right)^{-2} e^{2 i t \theta\left(k_{j}\right)} & 0
\end{array}\right), \quad j \in \Delta_{k_{0}}^{+}.
\end{array}\right.\\
\underset{k={k}^{*}_{j}}{\operatorname{Res}} M^{(1)}(k)=\left\{\begin{array}{ll}
\underset{k \rightarrow {k}^{*}_{j}}\lim M^{(1)}(k)\left(\begin{array}{cc}
0 & 0 \\
-({c}^{*}_{j})^{-1} {T^{\prime}}^{-2}({k}^{*}_{j}) e^{2 i t \theta({k}^{*}_{j})} & 0
\end{array}\right), \quad j \in \Delta_{k_{0}}^{-}, \\
\underset{k \rightarrow {k}^{*}_{j}}\lim  M^{(1)}(k)\left(\begin{array}{cc}
0 & -{c}^{*}_{j} T^{2}({k}^{*}_{j}) e^{-2 i t \theta({k}^{*}_{j})} \\
0 & 0
\end{array}\right), \quad j \in \Delta_{k_{0}}^{+} .
\end{array}\right.
\end{array}
\end{equation}
\end{rhp}
In addition, due to $T(k)^{-\sigma3}\rightarrow I$, as $k\rightarrow\infty$, so the solution of the coupled dispersion AB system can be expressed as
\begin{equation}
A= 4 i\lim _{k \rightarrow \infty} (k M^{(1)}(k))_{12},~~B=-\frac{4 i}{\beta} \lim _{k \rightarrow \infty} \frac{d}{d t}(k M^{(1)}(k))_{11}.
\end{equation}

\section{Contour deformation}
The next idea is to eliminate the jump on the real axis and open it at a small angle at the steady-state phase point to deform the contour of RHP \ref{4r}. According to the number of steady-state phase points and spectral singularity points, the following contours can be considered
\begin{equation}
\Sigma^{(2)}=\Sigma_{1} \cup \Sigma_{2}^{\pm} \cup \Sigma_{3}^{\pm} \cup \Sigma_{4} \cup \Sigma_{5} \cup \Sigma_{6}^{\pm} \cup \Sigma_{7}^{\pm} \cup \Sigma_{8},
\end{equation}
shown in Fig. \ref{4t}.
\begin{figure}[htpb]
\centering
\includegraphics[ width=8.6cm,height=3.3cm]{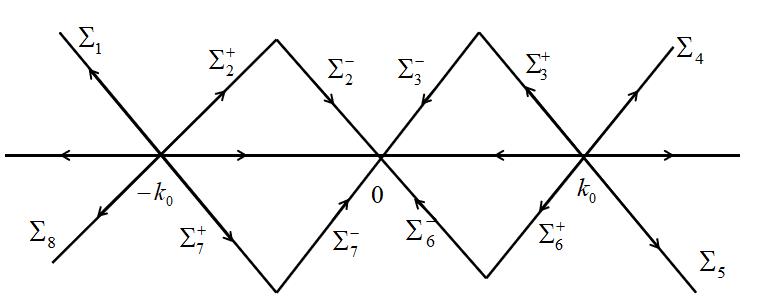}
\caption{\small Deformation from $\mathbb{R}$ to $\Sigma^{(2)}$.}\label{4t}
\end{figure}
In order to facilitate later applications, new tokens are introduced
$$
\varrho=\frac{1}{2} \min _{k, \mu \in \mathcal{K} \cup \mathcal{K}^{*}}|k-\mu|.
$$
Because of the symmetry, and the pole is not on the real axis, then for any $k_j=u_j+iv_j\in \mathcal{K}$, obviously there is $k^{*}_j=u_j-iv_j\in \mathcal{K}^{*}$. according to the above representation, there must be $|k_j-k^{*}_j|=2|v_j|>\varrho$. So there must be $\operatorname{dis}\left(k, \mathbb{R}\right)=\left|v\right| \geq \varrho>0$, here is due to the arbitrariness of $j$.

In order to keep the residual condition unchanged during contour deformation, an eigenfunction is defined near the discrete spectrum
\begin{equation}
\Upsilon_{\mathcal{K}}(k)= \begin{cases}1 , \operatorname{dist}\left(k, \mathcal{K} \cup \mathcal{K}^{*}\right)<\varrho / 3, \\ 0 , \operatorname{dist}\left(k, \mathcal{K} \cup \mathcal{K}^{*}\right)>2 \varrho / 3.\end{cases}
\end{equation}
Now we find a matrix $R_j \rightarrow \mathbb{C}$ with the following boundary conditions,
\begin{equation}\label{R1}
R_{1}(k)= \begin{cases}r(k) T^{-2}(k), & k \in\left(-\infty,-k_0 \right), \\ r\left(-k_0\right) T_{0}^{-2}\left(-k_0\right)\left(k+k_0\right)^{-2 i v\left(-k_0\right)}\left(1-\Upsilon_{\mathcal{K}}(k)\right), & k \in \Sigma_{1},\end{cases}
\end{equation}
\begin{equation}\label{R2}
R_{3}(k)= \begin{cases}\frac{r^{*}(k^{*}) T_{+}^{2}(k)}{1+|r(k)|^{2}}, & k \in\left(-k_0, k_0\right), \\ \frac{r^{*}\left(-k_0\right) T_{0}^{2}\left(-k_0\right)}{1+\left|r\left(-k_0\right)\right|^{2}}\left(k+k_0\right)^{2 i v\left(-k_0\right)}\left(1-\Upsilon_{\mathcal{K}}(k)\right), & k \in \Sigma^{+}_{2},\\
0,& k \in \Sigma^{-}_{2},\end{cases}
\end{equation}
\begin{equation}\label{R3}
R_{4}(k)= \begin{cases}\frac{r^{*}(k^{*}) T_{+}^{2}(k)}{1+|r(k)|^{2}}, & k \in\left(-k_0, k_0\right), \\ \frac{r^{*}\left(k_0\right) T_{0}^{2}\left(k_0\right)}{1+\left|r\left(k_0\right)\right|^{2}}\left(k-k_0\right)^{2 i v\left(k_0\right)}\left(1-\Upsilon_{\mathcal{K}}(k)\right), & k \in \Sigma_{3}^{+},\\
0,& k \in \Sigma^{-}_{3},
\end{cases}
\end{equation}
\begin{equation}\label{R4}
R_{5}(k)= \begin{cases}r(k) T^{-2}(k), & k \in\left(k_0,\infty \right), \\ r\left(k_0\right) T_{0}^{-2}\left(k_0\right)\left(k-k_0\right)^{-2 i v\left(k_0\right)}\left(1-\Upsilon_{\mathcal{K}}(k)\right), & k \in \Sigma_{4},\end{cases}
\end{equation}
\begin{equation}\label{R5}
R_{6}(k)= \begin{cases}r^{*}(k^{*}) T^{2}(k), & k \in\left(k_0,\infty \right),\\ r^{*}\left(k_0\right) T_{0}^{2}\left(k_0\right)\left(k-k_0\right)^{2 i v\left(k_0\right)}\left(1-\Upsilon_{\mathcal{K}}(k)\right), & k \in \Sigma_{5},\end{cases}
\end{equation}
\begin{equation}\label{R6}
R_{7}(k)= \begin{cases}\frac{r(k) T_{-}^{-2}(k)}{1+|r(k)|^{2}}, & k \in\left(-k_0, k_0\right), \\ \frac{r\left(k_0\right) T_{0}^{-2}\left(k_0\right)}{1+\left|r\left(k_0\right)\right|^{2}}\left(k-k_0\right)^{-2 i v\left(k_0\right)}\left(1-\Upsilon_{\mathcal{K}}(k)\right), & k \in \Sigma_{6}^{+},\\
0,& k \in \Sigma^{-}_{6},\end{cases}
\end{equation}
\begin{equation}\label{R7}
R_{8}(k)= \begin{cases}\frac{r(k) T_{-}^{-2}(k)}{1+|r(k)|^{2}}, & k \in\left(-k_0, k_0\right), \\ \frac{r\left(-k_0\right) T_{0}^{-2}\left(-k_0\right)}{1+\left|r\left(-k_0\right)\right|^{2}}\left(k+k_0\right)^{-2 i v\left(-k_0\right)}\left(1-\Upsilon_{\mathcal{K}}(k)\right), & k \in \Sigma_{7}^{+},\\
0,& k \in \Sigma^{-}_{7},\end{cases}
\end{equation}
\begin{equation}\label{R8}
R_{9}(k)= \begin{cases}r^{*}(k) T^{2}(k), & k \in\left(-\infty,\xi_{1}\right), \\ r^{*}\left(-k_0\right) T_{0}^{2}\left(-k_0\right)\left(k+k_0\right)^{2 i v\left(-k_0\right)}\left(1-\Upsilon_{\mathcal{K}}(k)\right), & k \in \Sigma_{8}.\end{cases}
\end{equation}

And meet the following estimates
\begin{equation}
\left|R_{j}\right| \lesssim \sin ^{2}\left(\arg \left(k\pm k_0\right)\right)+\langle\operatorname{Re}(k)\rangle^{-\frac{1}{2}}.~~j=1,3,...,9
\end{equation}

In this way, the original jump in $\pm k_0$ will become as shown in Fig. \ref{5t} and Fig. \ref{6t}.
\begin{figure}[h]
{\includegraphics[width=10cm,height=5.5cm]{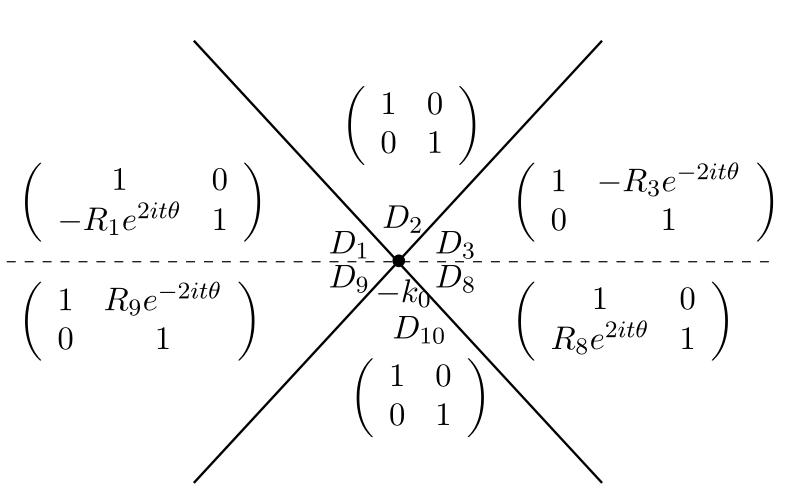}}
\centering
\caption{\small Matrix $\mathcal{R}^{(2)}$ around stationary phase point $-k_0$.}\label{5t}
\end{figure}

\begin{figure}[h]
{\includegraphics[width=10cm,height=5.5cm]{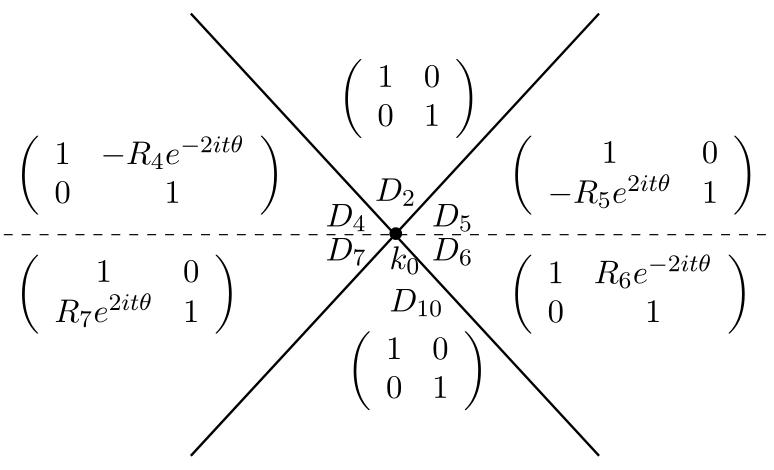}}
\centering
\caption{\small Matrix $\mathcal{R}^{(2)}$ around stationary phase point $k_0$.}\label{6t}
\end{figure}

In order to facilitate future calculation and estimation, the region needs to be divided appropriately, as shown in the Fig.\ref{7t} below.

\begin{figure}[h]
{\includegraphics[width=10cm,height=4.5cm]{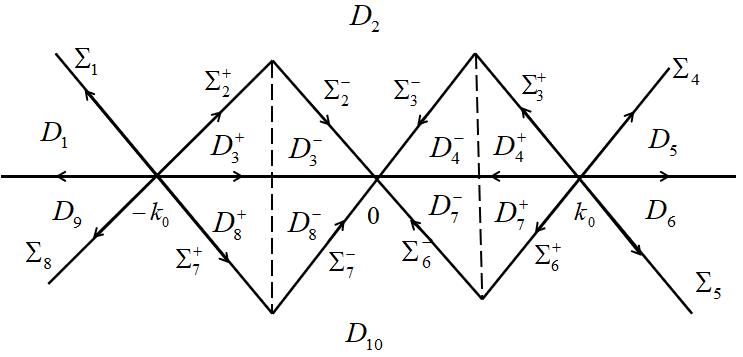}}
\centering
\caption{\small Region segmentation.}\label{7t}
\end{figure}
And making some marks
$$
\begin{array}{ll}
p_{1}(k)=p_{5}(k)=r(k), & p_{6}(k)=p_{9}(k)=r^{*}(k) \\
p_{3}(k)=p_{4}(k)=\frac{r^{*}(k)}{1+|r(k)|^{2}}, & p_{7}(k)=p_{8}(k)=\frac{r(k)}{1+|r(k)|^{2}},
\end{array}
$$
further estimates of $\bar{\partial}R_{j}$ in different regions can be obtained
\begin{lemma}
Suppose $r\in H^{1,1}(\mathbb{R})$, $\bar{\partial}R_{j}$ defined by (\ref{R1})-(\ref{R8}) satisfies
\begin{equation}
\left|\bar{\partial} R_{j}(k)\right| \lesssim\left|\bar{\partial} \Upsilon_{\mathcal{K}}(k)\right|+\left|p_j^{\prime}(\operatorname{Re}(k))\right|+\left|k\pm k_0\right|^{-\frac{1}{2}},
\end{equation}
 on $D_1,D_5,D_6,D_9$ and $D_3^{+},D_4^{+},D_7^{+},D_8^{+}$. $\bar{\partial}R_{j}$ satisfies
 \begin{equation}
\left|\bar{\partial} R_{j}(k)\right| \lesssim\left|\bar{\partial} \Upsilon_{\mathcal{K}}(k)\right|+\left|p_j^{\prime}(\operatorname{Re}(k))\right|+\left|k\right|^{-\frac{1}{2}}
\end{equation}
on $D_3^{-},D_4^{-},D_7^{-},D_8^{-}$.
\end{lemma}

Correspondingly, the jump line is transformed into $\Sigma^{(3)}$. Here, two jump lines are added, which are in the following form
\begin{equation}
\tilde{v}=\left(\begin{array}{ll}
I, & k \in\left(\frac{-ik_0}{2}\tan\frac{\pi}{24}, \frac{ik_0}{2}\tan\frac{\pi}{24}\right), \\
\left(\begin{array}{cc}
1 & \left(R_{2}^{+}-R_{2}^{-}\right) e^{-2 i \theta} \\
0 & 1
\end{array}\right), & k \in\left(\frac{ik_0}{2}\tan\frac{\pi}{24}-\frac{k_0}{2}, \frac{i k_0}{2\sqrt{3}} -\frac{k_0}{2}\right), \\
\left(\begin{array}{cc}
1 & 0 \\
\left(R_{7}^{+}-R_{7}^{-}\right) e^{2 i \theta} & 1
\end{array}\right) & k \in\left(-\frac{ik_0}{2}\tan\frac{\pi}{24}-\frac{k_0}{2}, -\frac{i k_0}{2\sqrt{3}} -\frac{k_0}{2}\right), \\
\left(\begin{array}{ccc}
1 & \left(R_{3}^{-}-R_{3}^{+}\right) e^{-2 i \theta} \\
0 & 1 &
\end{array}\right), & k \in\left(\frac{ik_0}{2}\tan\frac{\pi}{24}+\frac{k_0}{2}, \frac{i k_0}{2\sqrt{3}} +\frac{k_0}{2}\right), \\
\left(\begin{array}{ccc}
\left(R_{6}^{-}-R_{6}^{+}\right) e^{2 i \theta} & 1
\end{array}\right), & k \in\left(-\frac{ik_0}{2}\tan\frac{\pi}{24}+\frac{k_0}{2}, -\frac{i k_0}{2\sqrt{3}} +\frac{k_0}{2}\right).
\end{array}\right.
\end{equation}
The jump line $\Sigma^{(3)}$ is shown in Fig. \ref{8t}.

\begin{figure}[h]
{\includegraphics[width=9.6cm,height=3.5cm]{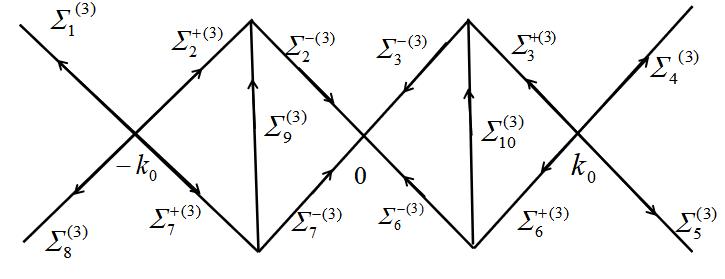}}
\centering
\caption{\small The contour of $\Sigma^{(3)}$.}\label{8t}
\end{figure}
At the same time, similar to Ref.\cite{CLL-2009-ar}, we can get $\tilde{v}$ subjects to the following estimate
$$
\left|\tilde{v}-I\right| \lesssim e^{-t}.
$$
Therefore, the following changes can be made
\begin{equation}
M^{(2)}(k)=M^{(1)}(k) \mathcal{R}^{(2)},
\end{equation}
where
$$
\mathcal{R}^{(2)}(k)= \begin{cases}\left(\begin{array}{cc}
1 & 0 \\
(-1)^{m_j}R_{j}(k) e^{2 i t \theta} & 1
\end{array}\right), & k \in D_{ j},\quad j=1,5,7,8 \\
\left(\begin{array}{cc}
1 & (-1)^{m_j}R_{j}(k) e^{-2 i t \theta} \\
0 & 1
\end{array}\right), & k \in D_{j},\quad j=3,4,6,9 \\
\left(\begin{array}{ll}
1 & 0 \\
0 & 1
\end{array}\right), & k \in D_{2} \cup D_{10},\end{cases}
$$
where $ m_1=m_3=m_4=m_5=1, m_6=m_7=m_8=m_9=0$. Matrix $M^{(2)}(k)$ satisfies the following RHP:
\begin{rhp}\label{5r}
The $M^{(2)}(k)$  obtained from $M^{(1)}(k)$  and $\mathcal{R}^{(2)}(k)$  above is expected to satisfy the following properties:

$\bullet$ $M^{(2)}$ is analytic within $\mathbb{C}\backslash \Sigma^{(2)}\cup\mathcal{K}\cup\mathcal{K}^{*};$

$\bullet$  $M^{(2)}(k)=\sigma_{0} (M^{(2)}(k^{*}))^{*}\sigma_{0}^{-1};$

$\bullet$$M^{(2)}(k)=I+\mathcal{O}\left(k^{-1}\right), \quad k \rightarrow \infty;$

$\bullet$ $M^{(2)}_{+}(k)=M^{(2)}_{-}(k) V^{(2)}(k), \quad k \in \Sigma^{(2)},$
where
$$
V^{(2)}(k)= \begin{cases}\left(\begin{array}{cc}
1 & 0 \\
R_{1}(k) e^{2 i t \theta} & 1
\end{array}\right), & k \in \Sigma_{1}, \\
\left(\begin{array}{cc}
1 & R_{3}(k) e^{-2 i t \theta} \\
0 & 1
\end{array}\right), & k \in \Sigma_{2}, \\
\left(\begin{array}{cc}
1 & R_{4}(k) e^{-2 i t \theta} \\
0 & 1
\end{array}\right), & k \in \Sigma_{3} ,\\
\left(\begin{array}{cc}
1 &  \\
R_{5}(k) e^{2 i t \theta} & 1
\end{array}\right), & k \in \Sigma_{4}, \\
\left(\begin{array}{cc}
1 & R_{6}(k) e^{-2 i t \theta}\\
0  & 1
\end{array}\right), & k \in \Sigma_{5},\\
\left(\begin{array}{cc}
1 & 0\\
R_{7}(k) e^{2 i t \theta}  & 1
\end{array}\right), & k \in \Sigma_{6},\\
\left(\begin{array}{cc}
1 & 0\\
R_{8}(k) e^{2 i t \theta}  & 1
\end{array}\right), & k \in \Sigma_{7},\\
\left(\begin{array}{cc}
1 & R_{9}(k) e^{-2 i t \theta} \\
0 & 1
\end{array}\right), & k \in \Sigma_{8}.
 \end{cases}
$$

$\bullet$  $M^{(2)}(k)$ has simple poles at each point in $\mathcal{K}\cup\mathcal{K}^{*}$ with:
\begin{equation}
\begin{array}{l}
\underset{k={k}_{j}}{\operatorname{Res}} M^{(2)}(k)=\left\{\begin{array}{ll}
\underset{k \rightarrow k_{j}}\lim M^{(2)}(k)\left(\begin{array}{cc}
0 & c_{j}^{-1}\left(\frac{1}{T}\right)^{\prime}\left(k_{j}\right)^{-2} e^{-2 i t \theta\left(k_{j}\right)} \\
0 & 0
\end{array}\right), \quad j \in \Delta_{k_{0}}^{-}, \\
\underset{k \rightarrow k_{j}}\lim  M^{(2)}(k)\left(\begin{array}{cc}
0 & 0 \\
c_{j} T\left(k_{j}\right)^{-2} e^{2 i t \theta\left(k_{j}\right)} & 0
\end{array}\right), \quad j \in \Delta_{k_{0}}^{+}.
\end{array}\right.\\
\underset{k={k}^{*}_{j}}{\operatorname{Res}} M^{(2)}(k)=\left\{\begin{array}{ll}
\underset{k \rightarrow {k}^{*}_{j}}\lim M^{(2)}(k)\left(\begin{array}{cc}
0 & 0 \\
-({c}^{*}_{j})^{-1} T^{\prime}({k}^{*}_{j})^{-2} e^{2 i t \theta({k}^{*}_{j})} & 0
\end{array}\right), \quad j \in \Delta_{k_{0}}^{-},\\
\underset{k \rightarrow {k}^{*}_{j}}\lim M^{(2)}(k)\left(\begin{array}{cc}
0 & -{c}^{*}_{j} T({k}^{*}_{j})^{2} e^{-2 i t \theta({k}^{*}_{j})} \\
0 & 0
\end{array}\right), \quad j \in \Delta_{k_{0}}^{+}.
\end{array}\right.
\end{array}
\end{equation}

$\bullet$  For $\mathbb{C}\backslash \Sigma^{(2)}\cup\mathcal{K}\cup\mathcal{K}^{*}$, we have the $\bar{\partial}$ derivative
\begin{equation}
\bar{\partial} M^{(2)}(k)=M^{(2)}(k) \bar{\partial} \mathcal{R}^{(2)}(k),\label{4.1}
\end{equation}
where
$$
\mathcal{R}^{(2)}(k)= \begin{cases}\left(\begin{array}{cc}
1 & 0 \\
(-1)^{m_j}\bar{\partial}R_{j}(k) e^{2 i t \theta} & 1
\end{array}\right), & k \in D_{ j},\quad j=1,5,7,8 \\
\left(\begin{array}{cc}
1 & (-1)^{m_j}\bar{\partial}R_{j}(k) e^{-2 i t \theta} \\
0 & 1
\end{array}\right), & k \in D_{j},\quad j=3,4,6,9 \\
\left(\begin{array}{ll}
0 & 0 \\
0 & 0
\end{array}\right), & k \in D_{2} \cup D_{10},\end{cases}
$$
where $m_1=m_3=m_4=m_5=1, m_6=m_7=m_8=m_9=0$.
\end{rhp}
The relationship between the solution of the coupled dispersion AB system and $M^{(2)}(k)$ is
\begin{equation}
A(x,t)= 4 i\lim _{k \rightarrow \infty} (k M^{(2)}(k))_{12},~~B(x,t)=-\frac{4 i}{\beta} \lim _{k \rightarrow \infty} \frac{d}{d t}(k M^{(2)}(k))_{11}.
\end{equation}
The $\bar{\partial}$ derivative appears in the above RHP, so it is also called mixed $\bar{\partial}$-$RHP$.

\section{Decomposition of the RHP \ref{5r}}
The following is mainly about the classification of RHP \ref{5r}. For the case of $\bar{\partial} \mathcal{R}^{(2)}(k)=0$, it is called a pure RH problem, and for the case of $\bar{\partial} \mathcal{R}^{(2)}(k)\neq0$, it is called a pure $\bar{\partial}$ problem. In the process of classification, consider the transformation $M^{(2)}(k)=M^{(3)}(k)M^{(2)}_{rhp}(k)$, if $\bar{\partial} \mathcal{R}^{(2)}(k)=0$, it corresponds to $M^{(2)}_{rhp}(k)$, if $\bar{\partial} \mathcal{R}^{(2)}(k)\neq0$, it corresponds to $M^{(3)}(k)=M^{(2)}(k)(M^{(2)}_{rhp}(k))^{-1}$. For $M^{(2)}_{rhp}(k)$, its jump is the same as $M^{(2)}(k)$, and its more properties are
\begin{rhp}\label{r6}
Find a matrix-valued function $M^{(2)}_{rhp}(k)$ with following properties:

$\bullet$ $M^{(2)}_{rhp}(k)$ is analytic within $\mathbb{C}\backslash \Sigma^{(2)}\cup\mathcal{K}\cup\mathcal{K}^{*};$

$\bullet$   Symmetry: $M^{(2)}_{rhp}(k)=\sigma_{0} (M^{(2)}_{rhp})^{*}(k^{*})\sigma_{0}^{-1};$

$\bullet$  Analytic behavior:
\begin{equation}M^{(2)}_{rhp}(k)=I+\mathcal{O}\left(k^{-1}\right), \quad k \rightarrow \infty,\end{equation}

$\bullet$  Jump condition: \begin{equation}
M^{(2)}_{rhp+}(k)=M^{(2)}_{rhp-}( k) V^{(2)}(k), \quad k \in \Sigma^{(2)};
\end{equation}
$\bullet$  Residue conditions: $M^{(2)}_{rhp}(k)$ has simple poles at each point in $\mathcal{K}\cup\mathcal{K}^{*}$ with
\begin{equation}
\begin{array}{l}
\underset{k=k_{j}}{\operatorname{Res}}M^{(2)}_{rhp}(k)=\left\{\begin{array}{ll}
\underset{k \rightarrow k_{j}}\lim M^{(2)}_{rhp}(k)\left(\begin{array}{cc}
0 & c_{j}^{-1}\left(\frac{1}{T}\right)^{\prime}\left(k_{j}\right)^{-2} e^{-2 i t \theta\left(k_{j}\right)} \\
0 & 0
\end{array}\right), \quad j \in \Delta_{k_{0}}^{-}, \\
\underset{k \rightarrow k_{j}}\lim M^{(2)}_{rhp}(k)\left(\begin{array}{cc}
0 & 0 \\
c_{j} T\left(k_{j}\right)^{-2} e^{2 i t \theta\left(k_{j}\right)} & 0
\end{array}\right), \quad j \in \Delta_{k_{0}}^{+}.
\end{array}\right.\\
\underset{k={k}^{*}_{j}}{\operatorname{Res}} M^{(2)}_{rhp}(k)=\left\{\begin{array}{ll}
\underset{k \rightarrow k^{*}_{j}}\lim M^{(2)}_{rhp}(k)\left(\begin{array}{cc}
0 & 0 \\
-\left({c}^{*}_{j}\right)^{-1} T^{\prime}\left({k}^{*}_{j}\right)^{-2} e^{2 i t \theta\left({k}^{*}_{j}\right)} & 0
\end{array}\right), \quad j \in \Delta_{k_{0}}^{-}, \\
\underset{k \rightarrow k^{*}_{j}}\lim  M^{(2)}_{rhp}(k)\left(\begin{array}{cc}
0 & -{c}^{*}_{j} T\left({k}^{*}_{j}\right)^{2} e^{-2 i t \theta\left({k}^{*}_{j}\right)} \\
0 & 0
\end{array}\right), \quad j \in \Delta_{k_{0}}^{+} .
\end{array}\right.\label{5.3}
\end{array}
\end{equation}

$\bullet$  $\bar{\partial}$-Derivative: $\bar{\partial} R^{(2)}(k)=0$ for $k \in \mathbb{C}$.

\end{rhp}
When $\bar{\partial} \mathcal{R}^{(2)}(k)\neq0$, we use the above $M^{(2)}_{rhp}(k)$ to construct a transformation: $M^{(3)}(k)=M^{(2)}(k)(M^{(2)}_{rhp}(k))^{-1}$, which is a  pure $\bar{\partial}$ problem. For $M^{(3)}(k)$, we have the following properties:
\begin{rhp}\label{7r}
Find a matrix-valued function $M^{(3)}(k)$ with following properties:

 $\bullet$  $M^{(3)}(k)$ is continuous in $\mathbb{C}$, continuous first partial derivatives in $\mathbb{C}\backslash \Sigma^{(2)}\cup\mathcal{K}\cup\mathcal{K}^{*};$

 $\bullet$   Symmetry: $M^{(3)}(k)=\sigma_{0} (M^{(3)}(k^{*}))^{*}\sigma_{0}^{-1};$

  $\bullet$   Jump condition: $M^{(3)}_{+}(k)=M^{(3)}_{-}(k) , \quad k \in \Sigma^{(2)}$;

$\bullet$  Analytic behavior: $M^{(3)}(k)=I+\mathcal{O}\left(k^{-1}\right), \quad k \rightarrow \infty$;

$\bullet$  $\bar{\partial}$-Derivative: $\bar{\partial} M^{(3)}(k)=M^{(3)}(k)  M^{(2)}_{rhp}(k) \bar{\partial} R^{(2)}(k) M^{(2)}_{rhp}(k)^{-1}.$

\end{rhp}
For the above proof, we can refer to Ref.\cite{YT-2021-ar}. Here we avoid repetition and do not present it again.

The next focus is to find $M^{(2)}_{rhp}(k)$. First, two open intervals are defined
\begin{equation}
\begin{aligned}
&\mathcal{A}_{1}=\left\{k:\left|k+k_0\right| \leq \min \left\{\frac{k_0}{2}, \rho / 3\right\} \triangleq \varepsilon\right\}, \\
&\mathcal{A}_{2}=\left\{k:\left|k-k_0\right| \leq \min \left\{\frac{k_0}{2}, \rho / 3\right\} \triangleq \varepsilon\right\}.
\end{aligned}
\end{equation}
Then $M^{(2)}_{rhp}(k)$ is divided into three parts
\begin{equation}
M^{(2)}_{rhp}(k)= \begin{cases}M^{err}(k) M^{o u t}(k), & k \notin\left\{\mathcal{A}_{1} \cup \mathcal{A}_{2}\right\}, \\ M^{err}(k) M^{\left(-k_0\right)}(k), & k \in \mathcal{A}_{1}, \\ M^{err}(k) M^{\left(k_0\right)}(k), & k \in \mathcal{A}_{2}.\end{cases}\label{mout}
\end{equation}
Due to $\operatorname{dist}(\mathcal{K} \cup \mathcal{K}^{*}, \mathbb{R})>\varrho$, $M^{(2)}_{rhp}(k)$ , $M^{\left(-k_0\right)}(k)$ and $M^{\left(k_0\right)}(k)$ have no poles in $\mathcal{A}_{1}$ and $\mathcal{A}_{2}$. The matrix $M^{(2)}_{rhp}(k)$  is divided into three parts by decomposition: one part can be called the external model RH problem, represented by $M^{o u t}(k)$, which can be solved directly by considering the standard RH problem without reflection potential. The other two parts are near the phase points $M^{\left(-k_0\right)}(k)$ and $M^{\left(k_0\right)}(k)$ , which can be matched to the known PC model, namely the parabolic cylinder model in $\mathcal{A}_{1}$ and $\mathcal{A}_{2}$, to be solved in Section 8.4. In addition, matrix $M^{err}(k)$ is an error function, which can be solved by the small norm RH problem in Section 8.5.

Let's define  $L_{\epsilon}$ for a fixed $\epsilon$
$$
\begin{gathered}
L_{\epsilon}=\left\{k: k=k_0+A k_0 e^{\frac{3 i \pi}{4} }, \epsilon \leq A \leq \frac{1}{\sqrt{2}}\right\} \\
\cup\left\{k: k=k_0+A k_0 e^{\frac{i \pi}{4}}, \epsilon \leq A \leq \infty\right\} \\
\cup\left\{k: k=-k_0+A k_0 e^{\frac{i \pi}{4} }, \epsilon \leq A \leq \frac{1}{\sqrt{2}}\right\} \\
\cup\left\{k: k=-k_0+A k_0  e^{\frac{3 i \pi}{4} }, \epsilon \leq A \leq \infty\right\}.
\end{gathered}
$$
Through the above definition, we can write the estimation of $V^{(2)}(k)$ as follows
\begin{equation}
\begin{aligned}\label{VI}
&\left\|V^{(2)}(k)-I\right\|_{L^{\infty}\left(\Sigma_{+}^{(2)} \cap \mathcal{A}_{2}\right)}=\mathcal{O}\left(e^{\frac{\sqrt{2}t}{4}\alpha\left|k -k_0\right|}\left(k_{0}^{-2}-|k|^{-2}\right)\right), \\
&\left\|V^{(2)}(k)-I\right\|_{L^{\infty}\left(\Sigma_{-}^{(2)} \cap \mathcal{A}_{1}\right)}=\mathcal{O}\left(e^{\frac{\sqrt{2}t}{4}\alpha\left|k +k_0\right|}\left(k_{0}^{-2}-|k|^{-2}\right)\right), \\
&\left\|V^{(2)}(k)-I\right\|_{L^{\infty}\left(\Sigma^{(2)} \backslash\left(\mathcal{A}_{1} \cup \mathcal{A}_{2}\right)\right)}=O\left(e^{-2 t \epsilon}\right),
\end{aligned}
\end{equation}
where the contours are defined by
$$
\Sigma_{+}^{(2)}=\Sigma_{3} \cup  \Sigma_{4} \cup \Sigma_{5} \cup \Sigma_{6}, \quad \Sigma_{-}^{(2)}=\Sigma_{1} \cup \Sigma_{2} \cup \Sigma_{7} \cup \Sigma_{8}.
$$
This  implies that the jump matrix $V^{(2)}(k)$ goes to $I$ on both $\Sigma^{(2)} \backslash\left(\mathcal{A}_{1} \cup \mathcal{A}_{2}\right)$.

$M^{o u t}(k)$ is the solution on the soliton region of $M^{(2)}(k)$, which is defined as no jump on $\mathbb{C}$, only discrete spectral points, and analytical in $\mathcal{A}_{1}\cup\mathcal{A}_{2}$ and outside discrete spectral points, that is, the following RHP
\begin{rhp}\label{8r}
Find a matrix-valued function $M^{out}(k)$ with following properties:

 $\bullet$  Analyticity:  $M^{out}(k)$ is  is analytical in $\mathbb{C}\backslash \Sigma^{(2)}\cup\mathcal{K}\cup\mathcal{K}^{*};$

 $\bullet$   Symmetry: $M^{out}(k)=\sigma_{0} (M^{out}(k^{*}))^{*}\sigma_{0}^{-1};$

$\bullet$  Analytic behavior: $M^{out}(k)=I+\mathcal{O}\left(k^{-1}\right), \quad k \rightarrow \infty$,

$\bullet$  Residue conditions:~$M^{out}(k)$ has simple poles at each point in $\mathcal{K}\cup\mathcal{K}^{*}$ satisfying the
same residue relations with (\ref{5.3}) of $M^{(2)}_{rhp}{(k)}$.
\end{rhp}
\section{Pure RHP and its asymptotic behavior}
In this segment, we will probe into the asymptotic behavior of the external soliton region and the internal non soliton region.
\subsection{External soliton solution region}
As we all know, solitons are generated when the reflection data is equal to 0, that is, $r(k) = 0$. At this time, the trace function is simplified to
\begin{equation}
s_{11}(k)=\prod_{j=1}^{N} \frac{k-k_{j}}{k-{k}^{*}_{j}},~~~s_{22}(k)=\prod_{j=1}^{N} \frac{k-{k}^{*}_{j}}{k-k_{j}}.
\end{equation}
And the jump matrix $V (k)=I$, then RHP \ref{r1} can be simplified to
\begin{rhp}\label{9r}
For a given scattering data $\mathcal{P}=\left\{\left(k_{j}, c_{j}\right)\right\}_{j=1}^{N}$, a matrix $M(k|\mathcal{P})$ can be found to satisfy:

$\bullet$  Analyticity: $M(k|\mathcal{P})$ is analytical in $\mathbb{C} \backslash\left(\Sigma^{(2)} \cup \mathcal{K} \cup \mathcal{K}^{*}\right)$;

$\bullet$   Symmetry: $M(k|\mathcal{P})=\sigma_{0} M^{*}(k^{*}|\mathcal{P})\sigma_{0}^{-1};$

$\bullet$   Asymptotic behaviors:
$M(k|\mathcal{P})=I+\mathcal{O}\left(k^{-1}\right), \quad k \rightarrow \infty,$

$\bullet$  Residue conditions: $M(k|\mathcal{P})$ has simple poles at each point in $k_j\in \mathbb{C}^{+}$ and $k_j^{*}\in \mathbb{C}^{-}$ with:
\begin{equation}
\begin{array}{l}
\underset{k=k_{j}}{\operatorname{Res}}M(k|\mathcal{P})=\underset{k \rightarrow k_{j}}\lim M(k|\mathcal{P})\left(\begin{array}{cc}
0 & 0 \\
c_{j} e^{2 i t \theta\left(k_{j}\right)} & 0
\end{array}\right), \\
\underset{k={k}_{j}^{*}}{\operatorname{Res}}M(k|\mathcal{P})=\underset{k \rightarrow k^{*}_{j}}\lim M(k|\mathcal{P})\left(\begin{array}{cc}
0 & -{c}_{j}^{*} e^{-2 i t \theta\left(k_j^{*}\right)} \\
0 & 0
\end{array}\right).
\end{array}
\end{equation}
\end{rhp}
The uniqueness of RHP \ref{9r} solution can be easily proved by using Liouville theorem.

In order to facilitate future research, we divide the scattering data into two parts. Note $\nabla\subseteq\{1,...,N\}$, and define
$$
{s_{11}}_{\nabla}(k)=\prod_{j\in\nabla} \frac{k-k_{j}}{k-{k}^{*}_{j}}.
$$
Next, make a modified transformation of $M(k|\mathcal{P})$ defined above as follows
\begin{equation}
M_{\nabla}(k \mid \mathcal{D})=M(k \mid \mathcal{P}) s_{11,\nabla}(k)^{\sigma_{3}},\label{ms}
\end{equation}
where the scattering data are given by
$$
\mathcal{D}=\left\{\left(k_{j}, c_{j}^{\prime}\right)\right\}_{j=1}^{ N}, \quad c_{j}^{\prime}=c_{j} s_{11,\nabla}(k)^{2}.
$$
Therefore, $M_{\nabla}(k \mid \mathcal{D})$ satisfies the following modified discrete RHP:
\begin{rhp}\label{10r}
For a given scattering data $\mathcal{D}=\{(k_{j}, c_{j}^{\prime})\}_{j=1}^{N}$, a matrix $M_{\nabla}(k \mid \mathcal{D})$ can be found to satisfy:

$\bullet$  Analyticity: $M^{\nabla}(k \mid \mathcal{D})$ is analytical in $\mathbb{C} \backslash\left(\Sigma^{(2)} \cup \mathcal{K} \cup \mathcal{K}^{*}\right)$;

$\bullet$   Symmetry: $M^{\nabla}(k \mid \mathcal{D})=\sigma_{0} M^{*}_{\nabla}(k^{*} \mid \mathcal{D})\sigma_{0}^{-1};$

$\bullet$   Asymptotic behaviors:
$
M_{\nabla}(k \mid \mathcal{D})=I+\mathcal{O}\left(k^{-1}\right), \quad k \rightarrow \infty;
$

$\bullet$  Residue conditions: $M_{\nabla}(k \mid \mathcal{D})$ has simple poles at each point in $k_j\in \mathbb{C}^{+}$ and $k_j^{*}\in \mathbb{C}^{-}$ with:
\begin{equation}\label{91}
\begin{array}{l}
\underset{k=k_{j}}{\operatorname{Res}} M_{\nabla}(k \mid \mathcal{D})=\left\{\begin{array}{ll}
\underset{k \rightarrow k_{j}}\lim M_{\nabla}(k \mid \mathcal{D})\left(\begin{array}{cc}
0 & \varpi_j^{-1}(s_{11,\nabla}^{\prime})^{-2}(k_{j})\\
0 & 0
\end{array}\right), j \in \nabla ,\\
\underset{k \rightarrow k_{j}}\lim M_{\nabla}(k \mid \mathcal{D})\left(\begin{array}{cc}
0 & 0 \\
\varpi_js_{11,\nabla}^2(k_j) & 0
\end{array}\right), \quad j \notin \nabla.
\end{array}\right.\\
\underset{k={k}^{*}_{j}}{\operatorname{Res}} M_{\nabla}(k \mid \mathcal{D})=\left\{\begin{array}{ll}
\underset{k \rightarrow {k}^{*}_{j}} \lim M_{\nabla}(k \mid \mathcal{D})\left(\begin{array}{cc}
0 & 0 \\
-\varpi_j^{*-1}(s^{*,\prime}_{11,\nabla})^{-2}(k^{*}_j) & 0
\end{array}\right), j \in \nabla, \\
\underset{k \rightarrow {k}^{*}_{j}}\lim M_{\nabla}(k \mid \mathcal{D})\left(\begin{array}{cc}
0 &-\varpi_j^{*} (s_{11,\nabla}^{*})^{2}(k^{*}_{j}) \\
0 & 0
\end{array}\right), j \notin \nabla .
\end{array}\right.
\end{array}
\end{equation}
where $\varpi_j=c_{j} e^{2 i t \theta\left(k_{j}\right)}$.
\end{rhp}

\begin{proposition}
If $ A_{s o l}(x,t)=A_{s o l}\left(x, t |\mathcal{D}\right), B_{s o l}(x,t)=B_{s o l}\left(x, t |\mathcal{D}\right) $ denote the N-soliton solution of the system (\ref{AB}), for the scattering data without reflection $\mathcal{D}=\{(k_{j}, c_{j}^{\prime})\}_{j=1}^{N}$, RHP \ref{10r} has a unique solution and
$$
\begin{aligned}
A_{s o l}\left(x, t |\mathcal{D}\right)&=4 i \lim _{k \rightarrow \infty}\left[k M_{\nabla}(k \mid \mathcal{D})\right]_{12}\\
&=4 i \lim _{k \rightarrow \infty}\left[k M(k \mid \mathcal{P})\right]_{12}=A_{s o l}\left(x,t|\mathcal{P}\right),
\end{aligned}
$$
$$
\begin{aligned}
B_{s o l}(x, t|\mathcal{D})&=-\frac{4 i}{\beta} \lim _{k \rightarrow \infty} \frac{d}{d t}[k M_{\nabla}(k \mid \mathcal{D})]_{11}
\\
&=-\frac{4 i}{\beta} \lim _{k \rightarrow \infty} \frac{d}{d t}[k M(k \mid \mathcal{P})]_{11}=B_{s o l}(x, t|\mathcal{P}).
\end{aligned}
$$
\end{proposition}
\subsection{Existence and uniqueness of solutions for external RH problems}

We see that $M^{out}(k)$ is a reflection soliton solution, and the reflection mainly comes from $T(k)$. In order to connect $M^{out}(x,t,k)$ with the case of non reflection scattering $\mathcal{D}=\{(k_{j}, c_{j}^{\prime})\}_{j=1}^{N}$, an idea is to take $\nabla=\triangle_{k_0}^{-}$ in Eq.(\ref{ms}).  So there are
$$
\begin{aligned}
&T(k)=\prod_{j \in \Delta^{-}_{k_{0}}} \frac{k-k^{*}_{j}}{k-k_{j}} \exp \left(i \int_{-k_0}^{k_0} \frac{\nu(s)}{s-k} d k\right)=s_{11,\Delta^{-}_{k_{0}}} \delta(k),\\
&T\left(k_{j}\right)^{-2}=s_{11,\Delta^{-}_{k_{0}}}(k_{j})^{2} \delta(k_{j})^{-2}, \quad(\frac{1}{T}  )^{\prime}\left(k_{j}\right)^{-2}=s_{11,\Delta^{-}_{k_{0}}}^{\prime}\left(k_{j}\right)^{-2} \delta\left(k_{j}\right)^{2}.
\end{aligned}
$$
The scattering data can be written as
\begin{equation}\label{db}
\widetilde{\mathcal{D}}=\left\{\left(k_{j}, \widetilde{c}_{j}\right)\right\}_{j=1}^{N}, \quad \widetilde{c}_{j}= \begin{cases}c_{j}^{-1} s_{11,\Delta_{k^{-}_{0}}}^{\prime}\left(k_{j}\right)^{-2} \delta\left(k_{j}\right)^{2}, & j \in \Delta_{k_{0}}^{-} \\ c_{j} s_{11,\Delta_{k^{-}_{0}}}\left(k_{j}\right)^{2} \delta\left(k_{j}\right)^{-2}, & j \notin \Delta_{k_{0}}^{-}\end{cases}.
\end{equation}
Therefore, RHP \ref{10r} is rewritten as
\begin{rhp}\label{11r}
For a given scattering data $\mathcal{\tilde{D}}=\left\{\left(k_{j}, \tilde{c}_{j}\right)\right\}_{j=1}^{N}$, a matrix $M_{\Delta_{k_{0}}^{-}}(k \mid \mathcal{\tilde{D}})$ can be found to satisfy:

$\bullet$  $M_{\Delta_{k_{0}}^{-}}(k \mid \mathcal{\tilde{D}})$  is analytical in $\mathbb{C} \backslash\left( \mathcal{K} \cup \mathcal{K}^{*}\right)$;

$\bullet$  $M_{\Delta_{k_{0}}^{-}}(k \mid \mathcal{\tilde{D}})=\sigma_{0} M^{*}_{\Delta_{k_{0}}^{-}}(k^{*} \mid \mathcal{\tilde{D}})\sigma_{0}^{-1};$

$\bullet$  $M_{\Delta_{k_{0}}^{-}}(k \mid \mathcal{\tilde{D}})=I+\mathcal{O}\left(k^{-1}\right), \quad k \rightarrow \infty$;

$\bullet$  $M_{\Delta_{k_{0}}^{-}}(k \mid \mathcal{\tilde{D}})$ has simple poles at each point in $k_j\in \mathbb{C}^{+}$ and $k_j^{*}\in \mathbb{C}^{-}$ with:

\begin{equation}\label{r101}
\begin{array}{l}
\underset{k=k_{j}}{\operatorname{Res}} M_{\Delta_{k_{0}}^{-}}(k \mid \mathcal{\tilde{D}})=\left\{\begin{array}{ll}
\underset{k \rightarrow k_{j}}\lim M_{\Delta_{k_{0}}^{-}}(k \mid \mathcal{\tilde{D}})\left(\begin{array}{cc}
0 & \Lambda^{-1}_j s_{11,\Delta_{k^{-}_{0}}}^{\prime-2}(k_j) \\
0 & 0
\end{array}\right) , j \in \Delta_{k_{0}}^{-},  \\
\underset{k \rightarrow k_{j}}\lim M_{\Delta_{k_{0}}^{-}}(k \mid \mathcal{\tilde{D}})\left(\begin{array}{cc}
0 & 0 \\
\Lambda_j s_{11,\Delta_{k^{-}_{0}}}^{2}(k_j)& 0
\end{array}\right),  j \notin \Delta_{k_{0}}^{-},
\end{array}\right.\\
\underset{k={k}^{*}_{j}}{\operatorname{Res}} M_{\Delta_{k_{0}}^{-}}(k \mid \mathcal{\tilde{D}})=\left\{\begin{array}{ll}
\underset{k \rightarrow {k}^{*}_{j}}\lim M_{\Delta_{k_{0}}^{-}}(k \mid \mathcal{\tilde{D}})\left(\begin{array}{cc}
0 & 0 \\
-\Lambda^{*-1}_j s^{*\prime-2}_{11,\Delta_{k_{0}}^{-}}(k^{*}_j) & 0
\end{array}\right),  j \in \Delta_{k_{0}}^{-},  \\
\underset{k \rightarrow {k}^{*}_{j}}\lim  M_{\Delta_{k_{0}}^{-}}(k \mid \mathcal{\tilde{D}})\left(\begin{array}{cc}
0 & -\Lambda_j s_{11,\Delta_{k_{0}}^{-}}^{*^-2}(k^{*}_{j})\\
0 & 0
\end{array}\right), j \notin \Delta_{k_{0}}^{-}.
\end{array}\right.
\end{array}
\end{equation}
where $\Lambda_j=c_{j}\delta^{-2} e^{2 i t \theta\left(k_{j}\right)}$.
\end{rhp}
So we have the following corollary
\begin{corollary}
There exists and unique solution for the RHP \ref{8r}, moreover
$$
M^{out}(x,t,k)=M_{\Delta_{k_{0}}^{-}}(x,t,k \mid \tilde{\mathcal{D}}),
$$
where scattering data $\tilde{\mathcal{D}}$ is given by (\ref{db}) and
$$
\begin{aligned}
A_{sol}\left(x, t|\mathcal{D}\right)&=4 i \lim _{k \rightarrow \infty}\left[k M^{out}(x,t,k \mid \mathcal{D})\right]_{12}\\
&=4 i \lim _{k \rightarrow \infty}\left[k M_{\Delta_{k_{0}}^{-}}(x,t,k \mid \tilde{\mathcal{D}})\right]_{12}=A_{s o l}\left(x, t|\tilde{\mathcal{D}}\right),\\
B_{sol}\left(x, t|\mathcal{D}\right)&=-\frac{4 i}{\beta} \lim _{k \rightarrow \infty} \frac{d}{d t}[k M^{out}(x,t,k \mid \mathcal{D})]_{11}\\
&=-\frac{4 i}{\beta} \lim _{k \rightarrow \infty} \frac{d}{d t}[k M_{\Delta_{k_{0}}^{-}}(x,t,k \mid \tilde{\mathcal{D}})]_{11}=B_{s o l}\left(x, t|\tilde{\mathcal{D}}\right).
\end{aligned}
$$
\end{corollary}

 \subsection{The long-time behavior of soliton solution}
This section mainly considers the asymptotic behavior of soliton solutions. On the basis of  the residue condition (\ref{r101}), let $N=1,k=\zeta+i\eta$, under the condition of no scattering, the soliton solution of the coupled
dispersion AB system is
\begin{equation}
\begin{aligned}
&A(x,t)=4\eta sech(2\eta(x+\frac{\alpha}{4|k|^2}t))e^{i\zeta(-2x+\frac{\alpha}{2|k|^2}t)},\\
&B(x,t)=-\frac{2\alpha\eta^2}{\beta|k|^2}sech(\eta(2x+\frac{\alpha}{2|k|^2}t))^2.
\end{aligned}
\end{equation}
It can be seen that the velocity of soliton solution is $v=-\frac{\alpha}{4|k|^2}.$
Suppose $x_1 < x_2$ with $x_1,x_2\in \mathbb{R}$ and $v_1 < v_2$ with $v_1,v_2\in \mathbb{R}^{-}$, a conical region can be defined as
$$
\mathcal{C}\left(x_{1}, x_{2}, v_{1}, v_{2}\right)=\left\{(x, t) \in \mathbb{R}^{2} \mid x=x_{0}+v t, x_{0} \in\left[x_{1}, x_{2}\right], v \in\left[v_{1}, v_{2}\right]\right\}.
$$

\begin{figure}
{\includegraphics[width=9.2cm,height=5cm]{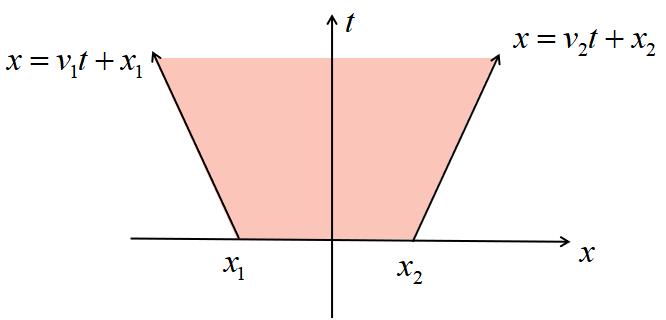}}
\centering
\caption{\small The cone $\mathcal{C}(x_1, x_2, v_1, v_2).$}\label{9t}
\end{figure}

Then make the following marks
\begin{equation}
\begin{aligned}
&\mathcal{I}=\left\{k: f\left(v_{2}\right)<|k|<f\left(v_{1}\right)\right\}, \quad f(v) \doteq\left(-\frac{\alpha}{4v}\right)^{\frac{1}{2}}, \\
&\mathcal{K}(\mathcal{I})=\left\{k_{j} \in \mathcal{K}: k_{j} \in \mathcal{I}\right\}, \quad N(\mathcal{I})=|\mathcal{K}(\mathcal{I})|, \\
&\mathcal{K}^{+}(\mathcal{I})=\left\{k_{j} \in \mathcal{K}:\left|k_{j}\right|>f\left(v_{1}\right)\right\}, \\
&\mathcal{K}^{-}(\mathcal{I})=\left\{k_{j} \in \mathcal{K}:\left|k_{j}\right|<f\left(v_{2}\right)\right\}, \\
&c_{j}^{\pm}(\mathcal{I})=c_{j} \prod_{\operatorname{Re} k_{n} \in I_{\pm} \backslash \mathcal{I}}\left(\frac{k_{j}-k_{n}}{k_{j}-k^{*}_{n}}\right)^{2} \exp \left[\pm \frac{1}{\pi i} \int_{I_{\pm}} \frac{\log [1+r(\varsigma) {r^{*}({\varsigma}^{*})]}}{\varsigma-k_{j}} d \varsigma\right].
\end{aligned}
\end{equation}
\begin{figure}[htpb]
\centering
{\includegraphics[width=7.5cm,height=5cm]{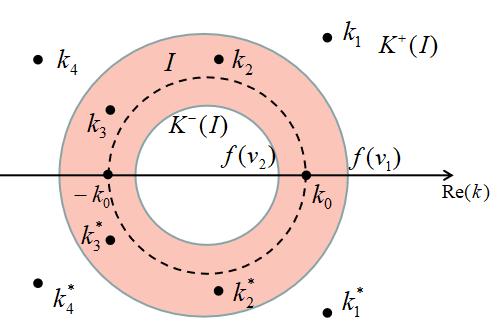}}

~~~~~~~~~~~~\caption{\small  Here, the original data has eight pairs zero points of discrete
spectrum, but insider the cone $\mathcal{C}$ only two pairs points with $\mathcal{K}(\mathcal{I})=\left\{k_{j} \in \mathcal{K}: k_{j} \in \mathcal{I}\right\},j=2,3$.}\label{t10}
\end{figure}

\begin{proposition}
Let $ \nabla=\Delta_{k_{0}}^{\mp}$ in $\mathcal{C}\left(x_{1}, x_{2}, v_{1}, v_{2}\right)$ on RHP \ref{10r}, and ensure the following estimation when $t\rightarrow\infty$:
\begin{equation}
\left\|\Theta_{j}^{\Delta^\mp_{k_{0}}}\right\|=\left\{\begin{array}{ll}
\mathcal{O}(1), & k_{j} \in \mathcal{K}(\mathcal{I}), \\
\mathcal{O}\left(e^{-2 \mu t}\right), & k_{j} \in \mathcal{K} \backslash \mathcal{K}(\mathcal{I}),
\end{array} \quad t \rightarrow \infty\right.
\end{equation}

where $$\mu=\min _{k_{j} \in \mathcal{K} \backslash \mathcal{K}(\mathcal{I})}\left\{\operatorname{Im} k_{j}\operatorname{dist}\left(v_{k_{j}}-v\right)\right\},~~
v_j=-\frac{\alpha}{4|k_j|^2},
$$
and in the formula, $v_j-v$ represents the difference between the speed of a soliton and the speed corresponding to $k$.
\end{proposition}
\begin{proof}
Here we take $\nabla=\triangle_{k_0}^{-}$ as an example to prove the above estimation. When $k=k_j\in\mathcal{K}^{-}(\mathcal{I})$ and $(x,t)\in \mathcal{C}(x_1,x_2,t_1,t_2)$, then the remainder condition for equation (\ref{91})
$$
|\varpi_j(x_0+vt,k_j)|\leq c|e^{-2i\theta(k_j)t}|.
$$
Next, focus on the index part
$$
-2i\theta(k_j)t=-2i(k_j\frac{x}{t}-\frac{\alpha}{4k_j})t=-2i(k_j\frac{x_0+vt}{t}-\frac{\alpha}{4k_j})t\\
=-2ik_jx_0-2i(k_jv-\frac{\alpha}{4k_j})t.
$$
Its real part is
$$
Re(-2i\theta(k_j)t)=2Imk_jx_0-2Imk_jt(v_j-v),
$$
so there
$$
|e^{-2i\theta(k_j)t}|=e^{2Imk_jx_0}e^{-2Imk_jt(v_j-v)}\leq ce^{-2\mu t}.
$$
\end{proof}
For the discrete spectrum $k_{j} \in \mathcal{K} \backslash \mathcal{K}(\mathcal{I})$, make a disk $D_{k}$ with a sufficiently small radius at each point and do not intersect each other, and define the function
\begin{equation}
\Gamma(k)= \begin{cases}I-\frac{1}{k-k_{j}} \Theta^{j}_{\Delta_{\mathcal{K}}^{\pm}}, & k \in D_{k}, \\ I-\frac{1}{k-{k^{*}_{j}}} \sigma_{2} \Theta^{*j}_{\Delta_{\mathcal{K}}^{\pm}} \sigma_{2}, & k \in \bar{D_{k}}, \\ I, & elsewhere \end{cases}
\end{equation}
where
$$
 \Theta^{j}_{\Delta_{\mathcal{K}}^{\pm}}=\begin{cases}\left(\begin{array}{cc}
0 & \varpi_j^{-1}(s_{11,\nabla}^{\prime})^{-2}(k_{j})\\
0 & 0
\end{array}\right), j \in \nabla ,\\
\left(\begin{array}{cc}
0 & 0 \\
\varpi_js_{11,\nabla}^2(k_j) & 0
\end{array}\right), \quad j \notin \nabla.
\end{cases}
$$
Make transformation
$$
\widetilde{M}_{\Delta_{k_{0}}^{\pm}}\left(k \mid \widetilde{D}_{\Delta_{k_{0}}^{\pm}}\right)=M_{\Delta_{k_{0}}^{\pm}}\left(k \mid \widetilde{D}_{\Delta_{k_{0}}^{\pm}}\right) \Gamma(k),
$$
and $\widetilde{M}_{\Delta_{k_{0}}^{\pm}}\left(k \mid \widetilde{D}_{\Delta_{k_{0}}^{\pm}}\right)$ satisfies the following relationship
\begin{equation}\label{M}
\widetilde{M}_{\Delta_{k_{0}}^{\pm}}\left(k \mid \widetilde{D}_{\Delta_{k_{0}}^{\pm}}\right)=M_{\Delta_{k_{0}}^{\pm}}\left(k \mid \widetilde{D}_{\Delta_{k_{0}}^{\pm}}\right) \widetilde{V}{(k)},k \in \widetilde{\Sigma}=\underset{k_{j} \in \mathcal{K} \backslash \mathcal{K}(\mathcal{I})}\bigcup\partial D_{k} \cup \partial \bar{D}_{k}.
\end{equation}
The jump matrix $\widetilde{V}(k)=\Gamma(k) $ is known based on the above proposition 8
\begin{equation}\label{vk}
\|\widetilde{V}(k)-I\|_{L^{\infty}(\widetilde{\Sigma})}=\mathcal{O}\left(e^{-2 \mu t}\right).
\end{equation}
In addition, let
\begin{equation}\label{m0}
M_0(k)=\widetilde{M}_{\Delta_{k_{0}}^{-}}\left(k \mid \widetilde{\mathcal{D}}_{\Delta_{k_{0}}^{-}}\right)\left[M_{\Delta^{-}_{k_0}}\left(k \mid \widetilde{\mathcal{D}}(\mathcal{I})\right)\right]^{-1},
\end{equation}
where $\widetilde{\mathcal{D}}(\mathcal{I})=\{k_j,c_j(\mathcal{I})\},~c_{j}(\mathcal{I})=c_{j} \prod_{\operatorname{Re} k_{n} \in I_{\pm} \backslash \mathcal{I}}\left(\frac{k_{j}-k_{n}}{k_{j}-{k}^{*}_{n}}\right)^{2}$.
it is obvious that $\widetilde{M}_{\Delta^{-}_{k_0}}(k|\widetilde{\mathcal{D}})$ and $M_{\Delta^{-}_{k_0}}(k|\widetilde{\mathcal{D}}) $ have the same residue condition, so $M_0(k)$ has no poles, its jump is
\begin{equation}\label{m0v}
M_{0}^{+}(k)=M_{0}^{-}(k) V_{M_{0}}(k),~~~~k\in \widetilde{\Sigma}.
\end{equation}
Further, the form of jump matrix $V_{M_{0}}(k)$ is
\begin{equation}
V_{M_{0}}(k)=M_{\Delta_{k_{0}}^{-}}\left(k \mid \widetilde{\mathcal{D}}(\mathcal{I})\right) \widetilde{V}(k) M_{\Delta_{k_{0}}^{-}}\left(k \mid \widetilde{\mathcal{D}}(\mathcal{I})\right)^{-1},
\end{equation}
according to equation (\ref{vk})
\begin{equation}
\left\|V_{M_{0}}(k)-I\right\|_{L^{\infty}(\widetilde{\Sigma})}=\|\tilde{V}(k)-I\|_{L^{\infty}(\widetilde{\Sigma})}=\mathcal{O}\left(e^{-2 \mu t}\right), \quad t \rightarrow \infty.
\end{equation}
Due to the small norm RH property, it is known that $M_0(k)$ exists and
$$
M_0(k)=I+\mathcal{O}\left(e^{-2 \mu t}\right), \quad t \rightarrow \infty.
$$
So there are
\begin{equation}
M_{\Delta_{k_0}^{\pm}}\left(k \mid \widetilde{\mathcal{D}}\right)=\left(I+\mathcal{O}\left(e^{-2 \mu t}\right)\right) M_{\Delta_{k_0}^{\pm}}\left(k \mid \widetilde{\mathcal{D}}(\mathcal{I})\right).
\end{equation}
\begin{corollary}
$\widetilde{\mathcal{D}}$ represents the scattering data of the equation (\ref{AB}), corresponding to the N-soliton solution of the coupled dispersion AB system, and $\widetilde{\mathcal{D}}(\mathcal{I})$ represents the scattering data on $\mathcal{K}(\mathcal{I})$, when $x,t\in \mathcal{C}(x_1,x_2,v_1,v_2)$
$$
\begin{aligned}
A_{s o l}\left(x, t|\widetilde{\mathcal{D}}\right)&=4 i \lim _{k \rightarrow \infty}\left[k M_{\Delta_{k_{0}}^{\pm}}(x,t,k \mid \widetilde{\mathcal{D}})\right]_{12}\\
&=4 i \lim _{k \rightarrow \infty}\left[k M_{\Delta_{k_{0}}^{\pm}(\mathcal{I})}(x,t,k \mid \widetilde{\mathcal{D}}({\mathcal{I}})\right]_{12}+\mathcal{O}\left(e^{-2 \mu t}\right), \quad t \rightarrow \infty,
\end{aligned}
$$
$$
\begin{aligned}
B_{s o l}\left(x, t|\widetilde{\mathcal{D}}\right)&=-\frac{4 i}{\beta} \lim _{k \rightarrow \infty}\frac{d}{dt}\left[k M_{\Delta_{k_{0}}^{\pm}}(x,t,k \mid \widetilde{\mathcal{D}})\right]_{11}\\
&=-\frac{4 i}{\beta} \lim _{k \rightarrow \infty}\frac{d}{dt}\left[k M_{\Delta_{k_{0}}^{\pm}(\mathcal{I})}(x,t,k \mid \widetilde{\mathcal{D}}({\mathcal{I}})\right]_{11}+\mathcal{O}\left(e^{-2 \mu t}\right), \quad t \rightarrow \infty.
\end{aligned}
$$

\end{corollary}
\subsection{Solvable RHP near phase point}
From Eq.(\ref{VI}), we can see that $||V(k)-I||$ has no consistent small jump in the neighborhood $\mathcal{A}_1$ and $\mathcal{A}_2$ when $t\rightarrow\infty$, so we need to establish a local model $M^{in}(k)$, which completely matches the jump of $M_{rhp}(k)$ on $(\Sigma^{(2)}\cap \mathcal{A}_1)\cup (\Sigma^{(2)}\cap \mathcal{A}_2)$ of the error function $M^{err}(k)$, so as to achieve a consistent estimation of the attenuation of the transition.
\begin{rhp}\label{12r}
For a matrix $M^{AB}(k)$, the following properties are satisfied:\\
~~~~~~~~$\bullet$  $M^{AB}(k)$ is analyticity in $\mathbb{C} \backslash \Sigma^{AB},$ $\Sigma^{AB}=\cup_{j}\Sigma^{j}, $ $j=1,2^{+},3^{+},4,5,6^{+}$,$7^{+},8;$

$\bullet$  $M^{AB}(k)=I+\mathcal{O}\left(k^{-1}\right), \quad k \rightarrow \infty$;

$\bullet$ $M^{AB}(k)$ has continuous boundary values $M^{AB}_{\pm}(k)$ on $\Sigma^{AB}$ and
$$
M_{+}^{AB}(k)=M_{-}^{AB}(k) V^{AB}|_{\mathcal{A}_1\cup \mathcal{A}_2}(k), \quad k \in \Sigma^{AB},
$$
where
$$
V^{AB}|_{\mathcal{A}_1\cup \mathcal{A}_2}(k)= \begin{cases}\left(\begin{array}{cc}
1 & 0 \\
r\left(-k_0\right) \delta^{-2}\left(-k_0\right)\left(k+k_0\right)^{-2 i v\left(-k_0\right)} e^{2 i t \theta} & 1
\end{array}\right), & k \in \Sigma_{1}, \\
\left(\begin{array}{cc}
1 & \frac{r^{*}\left(-k_0\right) \delta^{2}\left(-k_0\right)}{1+\left|r\left(-k_0\right)\right|^{2}}\left(k+k_0\right)^{2 i v\left(-k_0\right)}e^{-2 i t \theta} \\
0 & 1
\end{array}\right), & k \in \Sigma_{2}^{+}, \\
\left(\begin{array}{cc}
1 & \frac{r^{*}\left(k_0\right) \delta^{2}\left(k_0\right)}{1+\left|r\left(k_0\right)\right|^{2}}\left(k-k_0\right)^{2 i v\left(k_0\right)} e^{-2 i t \theta} \\
0 & 1
\end{array}\right), & k \in  \Sigma_{3}^{+}, \\
\left(\begin{array}{cc}
1 & 0  \\
r\left(k_0\right) \delta^{-2}\left(k_0\right)\left(k-k_0\right)^{-2 i v\left(k_0\right)} e^{2 i t \theta} & 1
\end{array}\right), & k \in \Sigma_{4}, \\
\left(\begin{array}{cc}
1 & r^{*}\left(-k_0\right) \delta^{2}\left(-k_0\right)\left(k+k_0\right)^{2 i v\left(-k_0\right)} e^{-2 i t \theta} \\
0 & 1
\end{array}\right), & k \in \Sigma_{8}, \\
\left(\begin{array}{cc}
1 & 0\\
\frac{r\left(-k_0\right) \delta^{-2}\left(-k_0\right)}{1+\left|r\left(-k_0\right)\right|^{2}}\left(k+k_0\right)^{-2 i v\left(-k_0\right)} e^{2 i t \theta}  & 1
\end{array}\right), & k \in \Sigma_{7}^{+},\\
\left(\begin{array}{cc}
1 & 0\\
\frac{r\left(k_0\right) \delta^{-2}\left(k_0\right)}{1+\left|r\left(k_0\right)\right|^{2}}\left(k-k_0\right)^{-2 i v\left(k_0\right)} e^{2 i t \theta}  & 1
\end{array}\right), & k \in \Sigma_{6}^{+},\\
\left(\begin{array}{cc}
1 & r^{*}\left(k_0\right)\delta^{2}\left(\xi_{2}\right)\left(k-k_0\right)^{2 i v\left(k_0\right)} e^{-2 i t \theta}\\
0  & 1
\end{array}\right), & k \in \Sigma_{5}. \end{cases}
$$
This is mainly because when there is no discrete spectrum, $T_0{(\pm k_0)}$ can be reduced to $\delta(\pm k_0)$ and $1-\Upsilon_{\mathcal{K}}(k)=1$.
\end{rhp}
Here, we need to consider the well-known PC model near the two phase points, as shown in Fig.\ref{10t}.
\begin{figure}[h]
{\includegraphics[width=9.5cm,height=2.8cm]{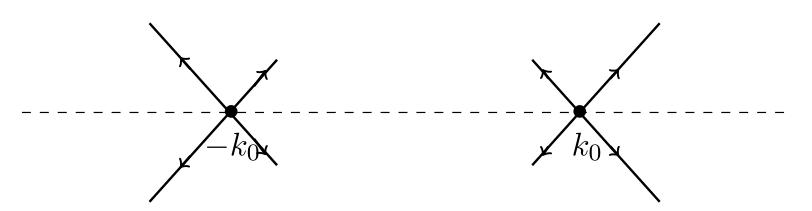}}
\centering
\caption{\small The jump contour Jump profile of local RHP near phase points $\pm k_0$}\label{10t}
\end{figure}

First, expand $\theta(k)$ at $k_0$ point
\begin{equation}
\theta(k)=-\frac{1}{2}\frac{\alpha}{k_0}-\frac{1}{4}\frac{\alpha}{k_0^3}(k-k_0)^2+\frac{1}{4}\frac{\alpha}{k_0^4}(k-k_0)^3,
\end{equation}

then define the following transformation
\begin{equation}\label{tf}
\mathbb{T}: f(k) \rightarrow(\mathbb{T }f)(k)=f\left(\sqrt{\frac{-k_0^3}{\alpha t}}\zeta+k_0\right),
\end{equation}
so there is on $\Sigma_4$
$$
\begin{aligned}
&r\left(k_0\right) \delta^{-2}\left(k_0\right)\left(k-k_0\right)^{-2 i v\left(k_0\right)} e^{2 i t \theta}\\
&=r\left(k_0\right) \delta^{-2}\left(k_0\right)(\sqrt{\frac{-k_0^3}{\alpha t}}\zeta_1)^{-2 i v\left(k_0\right)} e^{2 i t (-\frac{1}{2}\frac{\alpha}{k_0}-\frac{1}{4}\frac{\alpha}{k_0^3}(k-k_0)^2+\frac{1}{4}\frac{\alpha}{k_0^4}(k-k_0)^3)}\\
&=r\left(k_0\right) \delta^{-2}\left(k_0\right) e^{i \nu\left(k_0\right)\left(\ln (k_0^3)-\ln(-\alpha t)\right)}\zeta_1^{-2i\nu(k_0)}e^{\frac{i\zeta_1^2}{2}} e^{-i \alpha t  [\frac{1}{k_0}-\frac{1}{2}\frac{1}{k_0^4}(k-k_0)^3]},
\end{aligned}
$$
we set
\begin{equation}
r_{k_0}=r\left(k_0\right) \delta^{-2}\left(k_0\right) e^{i \nu\left(k_0\right)\left(\ln (k_0^3)-\ln(-\alpha t)\right)} e^{-i \alpha t  [\frac{1}{k_0}-\frac{1}{2}\frac{1}{k_0^4}(k-k_0)^3]},
\end{equation}
in the same way, we can calculate the values of $\Sigma_{2}^{+},\Sigma_{3}^{+},\Sigma_4,\Sigma_{6}^{+},\Sigma_{7}^{+}$ and $\Sigma_8$. Therefore, the jump matrix of $V^{AB}$ in $\mathcal{A}_1\cup \mathcal{A}_2$ can be rewritten as
$$
V^{AB}|_{\mathcal{A}_1\cup \mathcal{A}_2}(k)= \begin{cases}\left(\begin{array}{cc}
1 & 0 \\
r_{-k_0} \zeta^{-2 i \nu\left(-k_0\right)} e^{i \zeta^{2} / 2} & 1
\end{array}\right), & k \in \Sigma_{1}, \\
\left(\begin{array}{cc}
1 & \frac{{r}^{*}_{-k_0}}{1+\left|r_{-k_0}\right|^{2}} \zeta^{2 i \nu\left(-k_0\right)} e^{-i \zeta^{2} / 2} \\
0 & 1
\end{array}\right), & k \in \Sigma_{2}^{+}, \\
\left(\begin{array}{cc}
1 & \frac{{r}^{*}_{k_0}}{1+\left|r_{k_0}\right|^{2}} \zeta^{2 i \nu\left(k_0\right)} e^{-i \zeta^{2} / 2} \\
0 & 1
\end{array}\right), & k \in  \Sigma_{3}^{+}, \\
\left(\begin{array}{cc}
1 &  \\
r_{k_0} \zeta^{-2 i \nu\left(k_0\right)} e^{i \zeta^{2} / 2} & 1
\end{array}\right), & k \in \Sigma_{4}, \\
\left(\begin{array}{cc}
1 & {r}^{*}_{-k_0} \zeta^{2 i \nu\left(-k_0\right)} e^{-i \zeta^{2} / 2} \\
0 & 1
\end{array}\right), & k \in \Sigma_{8},\\
\left(\begin{array}{cc}
1 & 0\\
\frac{r_{-k_0}}{1+\left|r_{-k_0}\right|^{2}} \zeta^{-2 i \nu\left(-k_0\right)} e^{i \zeta^{2} / 2}  & 1
\end{array}\right), & k \in \Sigma_{7}^{+},\\
\left(\begin{array}{cc}
1 & 0\\
\frac{r_{k_0}}{1+\left|r_{k_0}\right|^{2}} \zeta^{-2 i \nu\left(k_0\right)} e^{i \zeta^{2} / 2}  & 1
\end{array}\right), & k \in \Sigma_{6}^{+},\\
\left(\begin{array}{cc}
1 & {r}^{*}_{k_0} \zeta^{2 i \nu\left(k_0\right)} e^{-i \zeta^{2} / 2}\\
0  & 1
\end{array}\right), & k \in \Sigma_{5}. \end{cases}
$$

If the PC model considered in $k_0$ is expanded in the following form
$$
M_{k_0}^{PC}(\zeta)=I+\frac{M_{1}^{p c}\left(k_0\right)}{i\zeta}+\mathcal{O}\left(\zeta^{-2}\right),
$$
where
$$
\begin{aligned}
&M_{1}^{PC}\left(k_0\right)=\left(\begin{array}{cc}
0 & \Xi_{12}\left(k_0\right) \\
-\Xi_{21}\left(k_0\right) & 0
\end{array}\right), \\
&\Xi_{12}\left(k_0\right)=\frac{\sqrt{2 \pi} e^{i \pi / 4} e^{-\pi \nu\left(k_0\right) / 2}}{r_1 \Gamma\left(-i \nu\left(k_0\right)\right)}, \\
&\Xi_{21}\left(k_0\right)=-\frac{\sqrt{2 \pi} e^{-i \pi / 4} e^{-\pi \nu\left(k_0\right) / 2}}{{r^{*}_1} \Gamma\left(i \nu\left(k_0\right)\right)}.
\end{aligned}
$$
The same is true for $-k_0$.
According to the transformation equation (\ref{tf}), the relationship between $M^{AB}$ and $M^{PC}$ is
$$
M^{AB}=I+\frac{1}{ i\zeta} ( M_{1}^{PC}\left(k_0\right)+M_{1}^{PC}\left(-k_0\right))+\mathcal{O}\left(\zeta^{-2}\right),
$$
further can be written as
\begin{equation}\label{mabi}
M^{AB}(x, t ,k)=I+\sqrt{\frac{-k_0^3}{\alpha t}} \frac{M_1^{P C}\left(k_0\right)}{k-k_0}+\sqrt{\frac{-k_0^3}{\alpha t}}\frac{M_1^{P C}\left(-k_0\right)}{k+k_0}+\mathcal{O}(\zeta^{-2}) .
\end{equation}
From this, we get a consistent estimate
\begin{equation}\label{mab}
\left|M^{AB}-I\right| \lesssim O\left(t^{-\frac{1}{2}}\right).
\end{equation}
Near the local circle of $\pm k_0$ have
$$
\left|\frac{1}{k\mp k_0}\right|<c,
$$
where $c$ is a constant.
Finally, we use $M^{AB}$ to define a local model in $\pm k_0$
$$
M^{(i n)}(k)=M^{o u t}(k) M^{AB}(k),
$$
it has the same jump matrix as $M_{rhp}$ and is a bounded function in ${\mathcal{A}_1}$ and ${\mathcal{A}_2}$.
\subsection{Small norm RH problem of $M^{err}(k)$}
In the disk ${\mathcal{A}_1}\cup{\mathcal{A}_2}$, $M^{PC}(k)$ is consistent with the jump matrix of $M^{(2)}_{rhp}(k)$, but $M^{out}(k)$ does not have a jump. Therefore, the matrix $M^{err}(k)$ defined by (\ref{mout}) erases the jump of $M^{(2)}_{rhp}(k)$ inside the disk ${\mathcal{A}_1}\cup {\mathcal{A}_2}$, and there is still a jump from $M^{(2)}_{rhp}(k)$ outside the disk $\mathbb{C}\setminus{\mathcal{A}_1\cup\mathcal{A}_2}$. Therefore, the jump path of $M^{err}(k)$ is
$$
\Sigma^{err}=\cup_{n=1}^{2} \partial \mathcal{A}_{n} \cup\left(\Sigma^{(2)} \backslash \cup_{n=1}^{2} \mathcal{A}_{n}\right),
$$
where $\partial \mathcal{A}_{n}$ is clockwise, see Fig.\ref{11t}. $M^{err}(k)$ can be directly verified to meet the following RHP:
\begin{figure}
{\includegraphics[width=9cm,height=3.23cm]{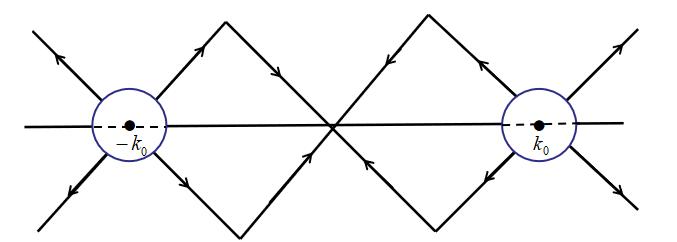}}
\centering
\caption{\small The jump contour $\Sigma^{err}$ for the $M^{err}(k)$.}\label{11t}
\end{figure}

\begin{rhp}\label{13r}
For a matrix $M^{err}(k)$, the following properties are satisfied:

$\bullet$  $M^{err}(k)$  is analytical in $\mathbb{C} \backslash \Sigma^{err}$;

$\bullet$  $M^{err}(k)=\sigma_{0} (M^{err}(k^{*}))^{*}\sigma_{0}^{-1};$

$\bullet$  $M^{err}(k)=I+\mathcal{O}\left(k^{-1}\right), \quad k \rightarrow \infty$;

$\bullet$ $M^{err}(k)$ has continuous boundary values $M_{\pm}^{err}(k)$ on $\Sigma^{err}$ and
$$
M^{err}_{+}(k)=M^{err}_{+}(k) V^{err}(k), \quad k \in \Sigma^{err},
$$
where
$$
V^{err}(k)= \begin{cases}M^{out}(k) V^{(2)}(k) M^{out}(k)^{-1}, & k \in \Sigma^{(2)} \backslash \cup_{n=1}^{2} \mathcal{A}_{n}, \\ M^ {out}(k) M^{AB}(k) M^{out}(k)^{-1}, & k \in \cup_{n=1}^{2} \partial \mathcal{A}_{n}.\end{cases}
$$
\end{rhp}
\begin{proof}
Here we mainly prove the form of $V^{err}(k)$. When $k \in \Sigma^{(2)} \backslash \cup_{n=1}^{2} \mathcal{A}_{n}$, $M^{err}(k)=M^{(2)}_{rhp}(k)(M^{out}(k))^{-1}$ and $M^{out}(k)$ does not jump in (\ref{mout}), then
$$
\begin{aligned}
M^{err}_{+}(k)&=M^{(2)}_{rhp+}(k)(M^{out}(k))^{-1}=M^{(2)}_{rhp-}(k)V^{(2)}(k)(M^{out}(k))^{-1}\\
&=M^{err}_{-}(k)V^{err}(k)=M^{2}_{rhp-}(k)M^{out}(k)^{-1}V^{err}(k) ,
\end{aligned}
$$
obviously,
$$
V^{err}=M^{out}(k)V^{(2)}(M^{out}(k))^{-1},~~~k \in \Sigma^{(2)} \backslash \cup_{n=1}^{2} \mathcal{A}_{n}.
$$
When $k \in\cup_{n=1}^{2} \mathcal{A}_{n}$, we have
$$M^{err}=M^{(2)}_{rhp}(k)(M^{(in)}(k))^{-1}=M^{(2)}_{rhp}(k)(M^{AB}(k))^{-1}(M^{out}(k))^{-1},$$ and $M^{(2)}_{rhp}(k)$ does not jump in $\cup_{n=1}^{2} \partial \mathcal{A}_{n}$, then
$$
M^{err}_{-}(k)=M^{(2)}_{rhp}(k) (M^{AB}(k))^{-1}(M^{out}(k))^{-1}.
$$
For $k \in \cup_{n=1}^{2} \partial \mathcal{A}_{n}$,
$$
M^{err}_{+}(k)=M^{(2)}_{rhp}(k) (M^{out}(k))^{-1},
$$
so from
$$M^{err}_{+}(k)=M^{err}_{-}(k)V^{err}(k),~~~~k \in \cup_{n=1}^{2} \partial \mathcal{A}_{n},$$
 we get
 $$
 V^{err}(k)=M^ {out}(k) M^{AB}(k) (M^{out}(k))^{-1},~~~k \in \cup_{n=1}^{2} \partial \mathcal{A}_{n}.
 $$
\end{proof}
Next, we will show that the small norm RH problem can be well solved by the error function $M^{err}(k)$ for large time.

It is known from equations $(\ref{itheta})$ and $(\ref{VI})$ that the jump matrix satisfies the following estimation
\begin{equation}\label{veii}
\left\|V^{err}-I\right\| \lesssim O\left(e^{-2 t \varepsilon}\right),~~~k \in \Sigma^{(2)} \backslash\left(\mathcal{A}_{1} \cup \mathcal{A}_{2}\right).
\end{equation}

In addition, it is known from equations (\ref{mab}) that when $k \in \cup_{n=1}^{2} \partial \mathcal{A}_{n} $, $V^{err}(k)$ satisfies the following estimation
\begin{equation}\label{vei}
\left\|V^{err}(k)-I\right\|=\left|M^{out }(k)^{-1}\left(M^{AB}(k)-I\right) M^{out}(k)\right|=\mathcal{O}\left(t^{-\frac{1}{2}}\right).
\end{equation}
\begin{proposition}
$RHP \ref{13r}$ has a unique solution in the form of
\begin{equation}
M^{err}(k)=I+\frac{1}{2 \pi i} \int_{\Sigma^{err}} \frac{\mu^{err}(s)\left(V^{err}(s)-I\right)}{s-k} d s,
\end{equation}
where $\mu^{err}\in L^{2}\left(\Sigma^{err}\right)$, meet $\left(1-C_{err}\right) \mu^{err}=I$, $C_{err}$ is Cauchy projection operator, defined as
$$
C_{err} f(z)=\lim _{k^{\prime} \rightarrow k \in \Sigma^{err}} \frac{1}{2 \pi i} \int_{\Sigma^{err}} \frac{f(s)}{s-k^{\prime}} d s.
$$
\end{proposition}
\begin{proof}
It can be seen from Eq.(\ref{vei}) that the operator is bounded. It can be seen from Ref.\cite{CF-2022-JDE}
\begin{equation}\label{ul}
\|\mu\|_{L^{2}\left(\Sigma^{err}\right)} \lesssim \frac{\left\|C_{err}\right\|}{1-\left\|C_{err}\right\|} \lesssim t^{-\frac{1}{2} },
\end{equation}
obviously, $1-C_{err}$ is reversible, so $M^{err}$ is unique.
\end{proof}
\begin{proposition}
Consider the asymptotic expansion of $M^{err}$ when $k\rightarrow\infty$
\begin{equation}
M^{err}(k)=I+\frac{M^{err}_{1}}{k}+\mathcal{O}\left(\frac{1}{k^{2}}\right),
\end{equation}
where
$$
\begin{aligned}
M^{err}_{1}&=\frac{1}{i}\sqrt{\frac{-k_0^3}{\alpha t}}M^{out}(k_0)M_1^{P C}\left(k_0\right)(M^{out})^{-1}(k_0)\\
&-\frac{1}{i}\sqrt{\frac{-k_0^3}{\alpha t}} M^{out}(-k_0)M_1^{P C}\left(-k_0\right)(M^{out})^{-1}(-k_0)+\mathcal{O}(t^{-1}).
\end{aligned}
$$
\end{proposition}
\begin{proof}
From Proposition 10,
$$
M^{err}_{1}=-\frac{1}{2 \pi i} \int_{\Sigma^{err}} \mu^{err}(s)\left(V^{err}(s)-I\right) d s ,
$$
then there are
\begin{equation}
\begin{aligned}
M^{err}_{1}&=-\frac{1}{2 \pi i} \int_{\Sigma^{err}}\left(V^{err}-I\right) d s-\frac{1}{2 \pi i} \int_{\Sigma^{err}}\left(\mu^{err}(s)-I\right)\left(V^{err}-I\right) d s \\
&=-\frac{1}{2 \pi i} \oint_{\cup_{n=1}^{2} \partial \mathcal{A}_{n} }\left(V^{err}-I\right) d s-\frac{1}{2 \pi i} \int_{\Sigma^{err} \backslash \cup_{n=1}^{2} \partial \mathcal{A}_{n} }\left(V^{err}-I\right) d s \\
&-\frac{1}{2 \pi i} \int_{\Sigma^{err}}\left(\mu^{err}(s)-I\right)\left(V^{err}-I\right) d s .
\end{aligned}
\end{equation}
Using equations (\ref{veii}), (\ref{ul}), $M^{err}_{1}$ can be written as
$$
M^{err}_{1}=-\frac{1}{2 \pi i} \oint_{\cup_{n=1}^{2} \partial \mathcal{A}_{n}}\left(V^{err}(s)-I\right) d s+O\left(t^{-1}\right).
$$
And using equations (\ref{mabi}) and (\ref{vei}), and residue theorem there are
$$
\begin{aligned}
M^{err}_{1}&=\frac{1}{i}\sqrt{\frac{-k_0^3}{\alpha t}}M^{out}(k_0)M_1^{P C}\left(k_0\right)(M^{out})^{-1}(k_0)\\
&-\frac{1}{i}\sqrt{\frac{-k_0^3}{\alpha t}} M^{out}(-k_0)M_1^{P C}\left(-k_0\right)(M^{out})^{-1}(-k_0)+\mathcal{O}(t^{-1}).
\end{aligned}
$$
\end{proof}
\section{Asymptotic analysis on the pure $\bar{\partial}$-problem}
In this section, we mainly solve the part of $\bar{\partial} \mathcal{R}^{(2)}(k)\neq0$ in $M^{(3)}(k)$, that is, RHP \ref{7r}. Its solution can be expressed in the following form
\begin{equation}\label{m3}
M^{(3)}(k)=I-\frac{1}{\pi} \iint_{\mathbb{C}} \frac{M^{(3)} W^{(3)}}{s-k} \mathrm{~d} A(s),
\end{equation}
where $W^{(3)}=M^{(2)}_{rhp}(k) \bar{\partial} R^{(2)} M^{(2)}_{rhp}(k)^{-1},$ and $dA(s)$ is the Lebesgue measure on the real plane. In fact, Eq.(\ref{m3}) can also be written in the form of an operator
\begin{equation}
(I-\mathbb{S}) M^{(3)}(k)=I,
\end{equation}
where $\mathbb{S}$  is Cauchy-Green operator,
$$
\mathbb{S}[f](k)=-\frac{1}{\pi} \iint_{\mathbb{C}} \frac{f(s) W(s)}{s-k} \mathrm{~d} A(s).
$$
If the operator $(I-\mathbb{S})^{-1}$ exists, then the above equation has a solution. Next, we will prove the existence of  operator $(I-\mathbb{S})^{-1}$ in regions $D_5$ and $D_4^{\pm}$, which can be proved by similar methods in other regions. Before proving the existence of operator $(I-\mathbb{S})^{-1}$, we give the following lemma

\begin{lemma}
\begin{equation}
\begin{aligned}
\left|\bar{\partial} R_{5}^{(2)} e^{2 i \theta t}\right| \lesssim\left(\left|\bar{\partial} \Upsilon_{\mathcal{K}}(s)\right|+\left|p_5^{\prime}(\operatorname{Re}(k))\right|+\left|k-k_0\right|^{-\frac{1}{2}}\right) e^{-c_1|u||v| t},~~~k \in D_{5}\\
\left|\bar{\partial} R_{7}^{(3)} e^{-2 i \theta t}\right| \lesssim\left\{\begin{array}{l}
\left(\left|\bar{\partial} \Upsilon_{\mathcal{K}}(s)\right|+\left|p_4^{\prime}(\operatorname{Re}(k))\right|+\frac{1}{|k|^{1/2}}\right) e^{-c_2|v| t},~~~k \in D_{4}^{-} \\
\left(\left|\bar{\partial} \Upsilon_{\mathcal{K}}(s)\right|+\left|p_4^{\prime}(\operatorname{Re}(k))\right|+\left|k-k_0\right|^{-\frac{1}{2}}\right) e^{-c_1|u||v| t}.~~~ k \in D_{4}^{+}
\end{array}\right.
\end{aligned}\label{l2}
\end{equation}
\end{lemma}

\begin{proof}
Recalling the above condition $\alpha<0$, for the convenience of analysis, the exponential part can be written as $-\frac{\alpha v t}{2}(\frac{1}{u^2+v^2}-\frac{1}{k_0^2})$, which is a decreasing function of $u$. In the $D_5$ region, the index part has
$$
\begin{aligned}
-\frac{\alpha v t}{2}\left(\frac{1}{(u+k_0)^2+v^2}-\frac{1}{k_0^2}\right)&=-\frac{\alpha v t}{2}\left(\frac{1}{(u+k_0)^2+v^2}-\frac{1}{k_0^2}\right)\\
&=-\frac{\alpha v t}{2}\left(\frac{-u^2-2k_0u-v^2}{((u+k_0)^2+v^2)(k_0^2)}\right)\\
&\leq -c_1|u||v|t.
\end{aligned}
$$
There is a similar analysis in area $D_4^{+}$. In region $D_4^{-}$, $k=u+iv$ and $0\leq v <u$ the index part has
$$
\begin{aligned}
\frac{\alpha v t}{2}(\frac{1}{u^2+v^2}-\frac{1}{k_0^2})&=\frac{\alpha v t}{2}(\frac{1}{u^2+v^2}-\frac{1}{k_0^2})\\
&\leq -c_2|v|t.
\end{aligned}
$$
The overall conclusion is summarized as equation (\ref{l2}).
\end{proof}
Therefore, our next goal is to prove the existence of $(I-\mathbb{S})^{-1}$.
\begin{proposition}
For sufficiently large $t$, operator $\mathbb{S}$ is a small norm and has
\begin{equation}\label{sl}
\|\mathbb{S}\|_{L^{\infty} \rightarrow L^{\infty}} \leqslant c t^{-\frac{1}{4}},
\end{equation}
therefore, $(I-\mathbb{S})^{-1}$ exists.
\end{proposition}
\begin{proof}
This is discussed in detail in the $D_5$ region, and other regions can be obtained similarly. Let $s=u+k_0+iv,k=\zeta+i\eta $, for any $f \in L^{\infty}$, we have
\begin{equation}
\begin{aligned}
|\mathbb{S}(f)| &\leq \frac{1}{\pi} \iint_{D_5} \frac{\left|f M_{rhp}^{(2)} \bar{\partial} R^{(2)} M_{rhp}^{(2)}{ }^{-1}\right|}{|s-k|} d A(s)\\
&\leq \frac{1}{\pi}\|f\|_{L^{\infty}}\left\|M_{rhp}^{(2)}\right\|_{L^{\infty}}\left\|M_{rhp}^{(2)}{ }^{-1}\right\|_{L^{\infty}} \int_{0}^{\infty} \int_{k_0+v}^{\infty}\frac{\left|\bar{\partial} R_{5}\right| e^{\frac{\alpha v t}{2}(\frac{1}{k_0^2}-\frac{1}{(u+k_0)^2+v^2})}}{|s-k|} du dv \\
&\leq c\left(\Omega_{1}+\Omega_{2}+\Omega_{3}\right),
\end{aligned}
\end{equation}
where
$$
\begin{aligned}
&\Omega_{1}=\int_{0}^{\infty} \int_{k_0+v}^{\infty}\frac{\left|\bar{\partial} \Upsilon_{\mathcal{K}}(s)\right|e^{\frac{\alpha v t}{2}(\frac{1}{k_0^2}-\frac{1}{(u+k_0)^2+v^2})}}{|s-k|}du dv,\\
&\Omega_{2}=\int_{0}^{\infty} \int_{k_0+v}^{\infty}\frac{\left|r^{\prime}(\operatorname{Re}(s))\right|e^{\frac{\alpha v t}{2}(\frac{1}{k_0^2}-\frac{1}{(u+k_0)^2+v^2})}}{|s-k|}du dv,\\
&\Omega_{3}=\int_{0}^{\infty} \int_{k_0+v}^{\infty}\frac{\left|s-k_0\right|^{-\frac{1}{2} }e^{\frac{\alpha v t}{2}(\frac{1}{k_0^2}-\frac{1}{(u+k_0)^2+v^2})}}{|s-k|} du dv.
\end{aligned}
$$

Next, the main idea is to estimate $\Omega_{j},j=1,2,3$. An important inequality to know in advance is
\begin{equation}\label{sk}
\begin{aligned}
\left\|\frac{1}{s-k}\right\|_{L^{2}\left(v+k_0, \infty\right)}^{2}&=\int_{v+k_0}^{\infty} \frac{1}{|s-k|^{2}} d u \leq \int_{-\infty}^{\infty} \frac{1}{|s-k|^{2}} d u \\
&=\int_{-\infty}^{\infty} \frac{1}{(u-\zeta)^{2}+(v-\eta)^{2}} d u=\frac{1}{|v-\eta|} \int_{-\infty}^{\infty} \frac{1}{1+y^{2}} d y=\frac{\pi}{|v-\eta|},
\end{aligned}
\end{equation}
where $y=\frac{u-\zeta}{v-\eta}$.

For $\Omega_1$, it can be obtained by direct calculation using Eq.(\ref{sk})
\begin{equation}
\begin{aligned}
\Omega_{1}
&\leq \int_{0}^{\infty} \left\|\bar{\partial} \Upsilon_{\mathcal{K}}(s)\right\|_{L^{2}}\left\|\frac{1}{s-k}\right\|_{L^{2}} d v \leq c_{3} \int_{0}^{\infty} \frac{e^{-\frac{\alpha v t}{2}(\frac{1}{(u+k_0)^2+v^2}-\frac{1}{k_0^2})}}{\sqrt{|\eta-v|}}d v\\
&\leq c_3 \int_{0}^{\infty} \frac{e^{-c_1|u||v|t}}{\sqrt{|\eta-v|}}d v \leq  c_3 \int_{0}^{\infty} \frac{e^{-c_1|v|^2 t}}{\sqrt{|\eta-v|}}d v\\
&\leq c t^{-\frac{1}{4}}.
\end{aligned}
\end{equation}
Use the same method to estimate $\Omega_2$. For the estimation of $\Omega_3$, Here we need the help of H\"{o}lder inequality with $ p> 2$ and $\frac{1}{p}+\frac{1}{q}=1$.
\begin{equation}\label{slp}
\begin{aligned}
\left\||s-k_0|^{-\frac{1}{2}}\right\|_{L^{p}(k_0+v,\infty)}&=\left(\int_{k_0+v}^{\infty} |u-k_0+i v|^{-\frac{p}{2}} d u\right)^{\frac{1}{p}} \\
&=\left(\int_{v}^{\infty} {|u+i v|^{-\frac{p}{2} }} d u\right)^{\frac{1}{p} }=\left(\int_{v}^{\infty} {\left(u^{2}+v^{2}\right)^{-\frac{p}{4}}} d u\right)^{\frac{1}{p}} \\
&=v^{\frac{1}{p}-\frac{1}{2}}\left(\int_{1}^{\infty} {\left(1+x^{2}\right)^{-\frac{p}{4}}} d x\right)^{\frac{1}{p}} \leq c v^{\frac{1}{p}-\frac{1}{2}}.
\end{aligned}
\end{equation}

Similar estimates can be proved
$$
\left\|\frac{1}{s-k}\right\|_{L^{q}(v, \infty)} \leq c|v-\eta|^{\frac{1}{q}-1}, \quad \frac{1}{q}+\frac{1}{p}=1.
$$
Then we have
\begin{equation}
\begin{aligned}
\Omega_{3} &\leq c \int_{0}^{+\infty}\left\||s-k_0|^{-\frac{1}{2}}\right\|_{L^{p}}\left\|\frac{1}{s-k}\right\|_{L^{q}}e^{-\frac{\alpha v t}{2}(\frac{1}{(u+k_0)^2+v^2})-\frac{1}{k_0^2}}d v\\
&\leq c \int_{0}^{+\infty}\left\||s-k_0|^{-\frac{1}{2}}\right\|_{L^{p}}\left\|\frac{1}{s-k}\right\|_{L^{q}}e^{-c_1|u||v|t}d v\\
&\leq c \int_{0}^{+\infty}\left\||s-k_0|^{-\frac{1}{2}}\right\|_{L^{p}}\left\|\frac{1}{s-k}\right\|_{L^{q}}e^{-c_1v^2t}d v\\
&\leq c\left[\int_{0}^{\eta}  v^{\frac{1}{p}-\frac{1}{2}}|v-\eta|^{\frac{1}{q}-1}e^{-c_1v^2t} d v+\int_{\eta}^{\infty} v^{\frac{1}{p}-\frac{1}{2}}|v-\eta|^{\frac{1}{q}-1}e^{-c_1v^2t} d v\right]\\
&\leq H_1+H_2,
\end{aligned}
\end{equation}
the following is the estimation of the two integrals. For $H_1$, since $0\leq v<\eta$, the following estimates are obtained by replacing $v=w\eta$ with variables
$$
\begin{aligned}
H_1=&\int_{0}^{\eta}  v^{\frac{1}{p}-\frac{1}{2}}(\eta-v)^{\frac{1}{q}-1} e^{-c_1v^2t}d v\\
&=\int_{0}^{1} \eta^{\frac{1}{2}}w^{\frac{1}{p}-\frac{1}{2}}(1-w)^{\frac{1}{q}-1} e^{-c_1w^2\eta^2t}d w \\
&\leq c t^{-\frac{1}{4}}.
\end{aligned}
$$
For $H_2$, we let $w=v-\eta$, then
$$
\begin{aligned}
H_2=&\int_{\eta}^{\infty}  v^{\frac{1}{p}-\frac{1}{2}}(v-\eta)^{\frac{1}{q}-1} e^{-c_1v^2t}d v\\
&=\int_{0}^{\infty} (w+\eta)^{\frac{1}{p}-\frac{1}{2}}w^{\frac{1}{q}-1} e^{-c_1(w+\eta)^2t}d w\\
&\leq c t^{-\frac{1}{4}}\int_{0}^{\infty} w^{-\frac{1}{2}} e^{-c_1(w)^2t}d w\\
&\leq c t^{-\frac{1}{4}}.
\end{aligned}
$$
So combining the estimates of $H_1$ and $H_2$, we get $\Omega_3\leq ct^{-\frac{1}{4}}$.

In the $D_4^{-}$ region, the estimation is still divided into three parts, at this time, there is $s=u+iv$.
\begin{equation}
\begin{aligned}
|\mathbb{S}(f)| &\leq \frac{1}{\pi} \iint_{D_4^{-}} \frac{\left|f M_{rhp}^{(2)} \bar{\partial} R^{(2)} M_{rhp}^{(2)}{ }^{-1}\right|}{|s-k|} d A(s)\\
&\leq \frac{1}{\pi}\|f\|_{L^{\infty}}\left\|M_{rhp}^{(2)}\right\|_{L^{\infty}}\left\|M_{rhp}^{(2)}{ }^{-1}\right\|_{L^{\infty}} \iint_{D_{4}^{-}} \frac{\left|\bar{\partial} R_{3}\right| e^{\frac{\alpha v t}{2}(\frac{1}{u^2+v^2}-\frac{1}{k_0^2})}}{|s-k|} d A(s) \\
&\leq c\left(\Omega^{\prime}_{1}+\Omega^{\prime}_{2}+\Omega^{\prime}_{3}\right),
\end{aligned}
\end{equation}
where
$$
\begin{aligned}
&\Omega^{\prime}_{1}=\int_{0}^{\frac{k_0}{2}} \int_{v}^{\frac{k_0}{2}}\frac{\left|\bar{\partial} \Upsilon_{\mathcal{K}}(s)\right|e^{\frac{\alpha v t}{2}(\frac{1}{u^2+v^2}-\frac{1}{k_0^2})}}{|s-k|}dudv,\\
&\Omega^{\prime}_{2}=\int_{0}^{\frac{k_0}{2}} \int_{v}^{\frac{k_0}{2}}\frac{\left|r^{\prime}(\operatorname{Re}(s))\right|e^{\frac{\alpha v t}{2}(\frac{1}{u^2+v^2}-\frac{1}{k_0^2})}}{|s-k|}dudv,\\
&\Omega^{\prime}_{3}=\int_{0}^{\frac{k_0}{2}} \int_{v}^{\frac{k_0}{2}}\frac{\left|s\right|^{-\frac{1}{2} }e^{\frac{\alpha v t}{2}(\frac{1}{u^2+v^2}-\frac{1}{k_0^2})}}{|s-k|} dudv.
\end{aligned}
$$
Combined with lemma (\ref{l2}) and Eq.(\ref{sk}), we can estimate each integral. For $\Omega^{\prime}_{1}$,
\begin{equation}
\begin{aligned}\label{o1}
\Omega^{\prime}_{1} & \lesssim \int_{0}^{\frac{k_0}{2}} e^{-c_2|v|t} \int_{v}^{\frac{k_0}{2}} \frac{r^{\prime}(u)}{|s-k|} d u d v \\
& \leq c \int_{0}^{\frac{k_0}{2}} \left\|\frac{1}{s-k}\right\|_{L^{2}}e^{-c_2|v|t}  d v\\
& \leq  c \int_{0}^{\infty} \frac{e^{-c_2|v|t}}{\sqrt{|\eta-v|}}d v\\
& \leq c t^{-\frac{1}{2}}.
\end{aligned}
\end{equation}
The estimation of $ \Omega^{\prime}_{2}$ is similar to that of $\Omega^{\prime}_{1}$, the estimation method of $\Omega^{\prime}_{3}$ is similar to that of $\Omega_{3}$, and  can get
\begin{equation}
\Omega^{\prime}_{j}\leq c t^{-\frac{1}{2}},~~~j=2,3.
\end{equation}
Therefore, this proves that all regions meet (\ref{sl}).
\end{proof}

Next, consider the expansion of $M^{(3)}$,
$$
M^{(3)}(k)=I+\frac{M_{1}^{(3)}(x, t)}{k}+\mathcal{O}\left(k^{-2}\right), \quad k \rightarrow \infty,
$$
from the previous equation (\ref{m3}), it is easy to know
$$
M_{1}^{(3)}(x, t)=\frac{1}{\pi} \int_{\mathbb{C}} M^{(3)}(s) W^{(3)}(s) d A(s).
$$
Further, we can prove that
\begin{proposition}
For a large t, we have
\begin{equation}\label{m13}
\left|M_{1}^{(3)}\right| \leq c t^{-\frac{3}{4} }.
\end{equation}
\end{proposition}

\begin{proof}
Owing to  $M_{rhp}^{(2)}$ is bounded outside the pole, there is
\begin{equation}
\begin{aligned}
\left|M_{1}^{(3)}\right| &\leq \frac{1}{\pi} \iint_{D_5} \left|M^{(3)} M_{rhp}^{(2)} \bar{\partial} R^{(2)} M_{rhp}^{(2)}{ }^{-1}\right| d A(s)\\
&\leq \frac{1}{\pi}\left\|M^{(3)}\right\|_{L^{\infty}}\left\|M_{rhp}^{(2)}\right\|_{L^{\infty}}\left\|\left(M_{rhp}^{(2)}\right)^{-1}\right\|_{L^{\infty}} \iint_{D_5}\left|\bar{\partial} R_{5} e^{2 i t \theta}\right| d A(s) \\
&\leq c(\iint_{D_{5}}|\bar{\partial}\Upsilon_{\mathcal{K}}(s)| e^{-\frac{\alpha v t}{2}(\frac{1}{(u+k_0)^2+v^2}-\frac{1}{k_0^2})} d A(s)+\iint_{D_5}|p_5^{\prime}(u)| e^{-\frac{\alpha v t}{2}(\frac{1}{(u+k_0)^2+v^2}-\frac{1}{k_0^2})} d A(s)\\
&~~~~~+\iint_{D_{5}} \frac{1}{\left|s-k_0\right|^{1 / 2}} e^{-\frac{\alpha v t}{2}(\frac{1}{(u+k_0)^2+v^2}-\frac{1}{k_0^2})} d A(s)) \\
&\leq c\left(\Omega_{4}+\Omega_{5}+\Omega_{6}\right).
\end{aligned}
\end{equation}
We constrain $\Omega_{4}$ by using the Cauchy-Schwarz inequality
$$
\begin{aligned}
\left|\Omega_{4}\right| \leq & \int_{0}^{\infty}\left\|\bar{\partial} \Upsilon_{\mathcal{K}}\right\|_{L_{u}^{2}\left(v+k_0, \infty\right)}\left(\int_{v}^{\infty}e^{-c_1 t |u| |v|}  d u\right)^{\frac{1}{2}} d v \\
\leq & c t^{-\frac{1}{2}} \int_{0}^{\infty} \frac{e^{-4 t v^{2}}}{\sqrt{v}} dv \\
\leq & c t^{-\frac{1}{4}} \int_{0}^{\infty} \frac{e^{-4 w^{2}}}{\sqrt{w}} d w \leq c t^{-\frac{3}{4}}.
\end{aligned}
$$
Similar constraints can be done for $\Omega_{5}$. For the constraint of $\Omega_{6}$, we follow the method of $\Omega_{3}$ and use H\"{o}lder  inequality and Eq.(\ref{slp}) to obtain
$$
\begin{aligned}
I_{3}=& \int_{0}^{+\infty} \int_{k_0+v}^{+\infty}\left(\left(u-k_0\right)^{2}+v^{2}\right)^{-\frac{1}{4}} e^{-\frac{\alpha v t}{2}(\frac{1}{(u+k_0)^2+v^2}-\frac{1}{k_0^2})}  d u d v \\
& \leq \int_{0}^{+\infty}\left\|\left(\left(u-k_0\right)^{2}+v^{2}\right)^{-1 / 4}\right\|_{L^{p}}\left(\int_{k_0+v}^{+\infty} e^{-c_1q |u| |v| t}du \right)^{\frac{1}{q}} d v\\
& \leq \int_{0}^{+\infty}v^{\frac{1}{p}-\frac{1}{2}}\left(\int_{k_0+v}^{+\infty} e^{-c_1q |u| |v| t}du \right)^{\frac{1}{q}} d v\\
&\leq c t^{-\frac{1}{q} } \int_{0}^{+\infty} v^{\frac{2}{p}-\frac{3}{2}} e^{-c_1q |v+k_0| |v| t} d v,\\
&\leq c t^{-\frac{3}{4} } \int_{0}^{\infty} w^{\frac{2}{p} -\frac{3}{2} } e^{-c_1q w^2 t} d w \leq c t^{-\frac{3}{4}}
\end{aligned}
$$
Here we replace the variable $v=wt^{-\frac{1}{2}}$. Notice that here $2<p<4$ and $-1<\frac{2}{p} -\frac{3}{2}<-\frac{1}{2}$.
\end{proof}

\section{Long time asymptotic behavior of soliton solution region for the coupled dispersive AB
system}
After many deformations above, we now begin to construct the long-term asymptotic properties of the coupled dispersive AB
system (\ref{AB}). Looking back at the previous transformation, we have
\begin{equation}\label{mt}
M(k)=M^{(3)}(k) M^{err}(k) M^{out}(k) \mathcal{R}^{(2)}(k)^{-1} T(k)^{\sigma_{3}}, \quad k \in \mathbb{C} \backslash\mathcal{A}_1\cup\mathcal{A}_2.
\end{equation}
In particular, in the vertical direction $k\in D_2,D_{10}$, there is $\mathcal{R}^{(2)} = I$, so we consider $k\rightarrow\infty$ in this region, so
$$
M=\left(I+\frac{M_{1}^{(3)}}{k}+\ldots\right)\left(I+\frac{M^{err}_{1}}{k}+\ldots\right)\left(I+\frac{M_{1}^{out}}{k}
+\ldots\right)\left(I+\frac{T_{1}^{\sigma_{3}}}{k}+\ldots\right).
$$
In order to recover the potential, the coefficient of ${k}^{-1}$ needs to be collected
$$
M_{1}=M_{1}^{o u t}+M^{err}_{1}+M_{1}^{(3)}+T_{1}^{\sigma_{3}},
$$
so there
\begin{equation}
\begin{aligned}
&A= 4 i(M_{1}^{o u t}+M_{1}^{err})_{12}+\mathcal{O}(t^{-\frac{3}{4}}),\\
&B=-\frac{4 i}{\beta}\frac{d}{dt} (M_{1}^{o u t}+M_1^{err})_{11}+\mathcal{O}(t^{-\frac{3}{4}}).
\end{aligned}
\end{equation}

Based on the above formula, the specific asymptotic state can be written in the form of the following theorem
\begin{theorem}
Suppose $A_0,B_0\in H^{1,1}(\mathbb{R})$ have general scattering data, and $A(x,t)$ and $B(x,t)$ are the solutions of equation (\ref{AB}). For fixed $x_1 < x_2$ with $x_1,x_2\in \mathbb{R}$ and $v_1 < v_2$ with $v_1,v_2\in \mathbb{R}^{-}$, a conical region can be defined as
$$
\mathcal{C}\left(x_{1}, x_{2}, v_{1}, v_{2}\right)=\left\{(x, t) \in \mathbb{R}^{2} \mid x=x_{0}+v t, x_{0} \in\left[x_{1}, x_{2}\right], v \in\left[v_{1}, v_{2}\right]\right\}.
$$
Then we define two regions for the spectral parameters
$$\mathcal{I}=\left\{k: f\left(v_{2}\right)<|k|<f\left(v_{1}\right)\right\}, \quad f(v) \doteq\left(-\frac{\alpha}{4v}\right)^{1 / 2},$$
as is shown in Fig. \ref{t10}. Using $A_{s o l}(x, t,k|\widetilde{\mathcal{D}})$ and $B_{s o l}(x, t,k|\widetilde{\mathcal{D}})$ for the coupled dispersive AB system corresponding to $N(\mathcal{I}) \leq N$  modulated non reflection scattering data
$$
\widetilde{D}(\mathcal{I})=\{k_j,c_j(\mathcal{I})\},~c_{j}(\mathcal{I})=c_{j}\prod_{k_n\in \mathcal{K}( \mathcal{I})}\left(\frac{k_{j}-k_{n}}{k_{j}-{k}^{*}_{n}}\right)^{2}.
$$
Then when $t\rightarrow\infty$ and $(x, t)\in \mathcal{C} (x_1, x_2, v_1, v_2)$, there is
$$
\begin{aligned}
&A(x, t)=A_{s o l}(x, t|\widetilde{D}(\mathcal{I}))+t^{-\frac{1}{2}}(g_1+g_2)+\mathcal{O}(t^{-\frac{3}{4}}),\\
&B(x, t)=B_{s o l}(x, t|\widetilde{D}(\mathcal{I}))+t^{-\frac{1}{2}}(h_1+h_2+f_1+f_2)+\mathcal{O}(t^{-\frac{3}{4}}),
\end{aligned}
$$
where
$$
M^{out}\left(\pm k_0\right)=\left(\begin{array}{cc}
m_{11}(\pm k_0) & m_{12}(\pm k_0) \\
m_{21}(\pm k_0) & m_{22}(\pm k_0)
\end{array}\right), ~j=1,2\\
$$
$$
\begin{aligned}
&g_1=4\sqrt{\frac{-k_0^3}{\alpha}}(m^2_{12}(-k_0)\Xi_{21}(-k_0)+m^2_{11}(-k_0)\Xi_{12}(-k_0)),\\
&g_2=4\sqrt{\frac{-k_0^3}{\alpha}}(m^2_{12}(k_0)\Xi_{21}(k_0)+m^2_{11}(k_0)\Xi_{12}(k_0)),\\
&h_1=\frac{4}{\beta}\sqrt{\frac{-k_0^3}{\alpha}}\frac{d}{dt}(m_{12}(-k_0)m_{22}(-k_0)\Xi_{21}(-k_0)+m_{11}(-k_0)m_{21}(-k_0)\Xi_{12}(-k_0)),\\
&h_2=\frac{4}{\beta}\sqrt{\frac{-k_0^3}{\alpha}}\frac{d}{dt}(m_{12}(k_0)m_{22}(k_0)\Xi_{21}(k_0)+m_{11}(k_0)m_{21}(k_0)\Xi_{12}(k_0)),\\
&f_1=\frac{2}{\beta t}\sqrt{\frac{-k_0^3}{\alpha}}m_{12}(-k_0)m_{22}(-k_0)\Xi_{21}(-k_0)+m_{11}(-k_0)m_{21}(-k_0)\Xi_{12}(-k_0),\\
&f_2=\frac{2}{\beta t}\sqrt{\frac{-k_0^3}{\alpha}}m_{12}(k_0)m_{22}(k_0)\Xi_{21}(k_0)+m_{11}(k_0)m_{21}(k_0)\Xi_{12}(k_0).
\end{aligned}
$$
\end{theorem}

\section*{Appendix A: Estimation of characteristic function}

In the previous assumption 2.1, there is no zero point of $s_{11}$ on $\mathbb{R}$, so there is no pole of $r(k)$ on $\mathbb{R}$.
From Volterra integral formula (\ref{jost})
$$
\begin{aligned}
&\psi_{11}^{\pm}(x,0,k)=1+\int_{\pm\infty}^{x}\frac{1}{2}A(y,k)\psi_{21}^{\pm}dy,\\
&\psi_{21}^{\pm}(x,0,k)=\int_{\pm\infty}^{x}-\frac{1}{2}A^{*}(y,k)\psi_{11}^{\pm}e^{2ik(x-y)}dy.
\end{aligned}
$$
Let
$$
\Psi_1^{-}=(\psi_{11}^{-},~~\psi_{21}^{-})^{T}.
$$
For $k \in \mathbb{R}$, we introduce an operator mapping 
$$
\Re_0 f=\int_{-\infty}^{x}T_0f(y)dy,
$$
and
$$
T_0=\left(\begin{array}{rr}
0 & \frac{1}{2}A \\
-\frac{1}{2}A^{*}e^{2ik(x-y)} & 0
\end{array}\right).
$$
It can be seen from this that
$$
|\Re_0f(x)| \leq \int_{-\infty}^{x}\frac{1}{2}\left|A(y)\right| d y\|f\|_{\left(L^{\infty}(-\infty, 0]\right)},
$$
which means that $\Re_0$ is a bounded linear operator on $L^{\infty}(-\infty, 0]$. Further, it can be obtained by mathematical induction that
$$
\left|\Re_0^{n} f(x)\right| \leq \frac{1}{n !}\left(\int_{-\infty}^{x}\frac{1}{2}\left|A(y)\right| d y\right)^{n}\|f\|_{\left(L^{\infty}(-\infty, 0]\right)}.
$$
Therefore, the following series is uniformly convergent on $x\in(-\infty, 0]$,
\begin{equation}
\Psi_1^{-}(x,k)=\sum_{n=0}^{\infty} \Re_0^{n}\left(\begin{array}{l}
1 \\
0
\end{array}\right).\label{p1}
\end{equation}
So for $k \in \mathbb{R}$  and $x\in(-\infty, 0]$,
$$
\left|\Psi_{1}(x, k)^{-}_{1}\right| \leq e^{\frac{1}{2}\left\|A\right\|_{L^{1}(\mathbb{R})}}.
$$
Similarly, we can get
$$
\begin{aligned}
&\left|\Psi^{-}_{2}(x,k)\right| \leq e^{\frac{1}{2}\left\|A\right\|_{L^{1}(\mathbb{R})}}, x \in(-\infty, 0]; \\
&\left|\Psi^{+}_{1}(x,k)\right| \leq e^{\frac{1}{2}\left\|A\right\|_{L^{1}(\mathbb{R})}}, x \in[0, \infty); \\
&\left|\Psi^{+}_{2}(x,k)\right| \leq e^{\frac{1}{2}\left\|A\right\|_{L^{1}(\mathbb{R})}}, x \in[0, \infty).
\end{aligned}
$$
Let's look at the derivative of the characteristic function with respect to $k$. According to the uniform convergence property of series (\ref{p1}), we have
$$
\begin{aligned}
\Psi^{-}_{1,k}(x, k) &=\sum_{n=0}^{\infty} \partial_{k}\Re_0^{n}\left(\begin{array}{l}
1 \\
0
\end{array}\right)=\sum_{n=1}^{\infty} \sum_{m=0}^{n} \Re_0^{m} \Re_0^{\prime} \Re_0^{n-m-1}\left(\begin{array}{l}
1 \\
0
\end{array}\right)
\end{aligned},
$$
where the form of $\Re_0^{\prime}$  is
$$
\Re_0^{\prime}f(x)=\int_{-\infty}^{x}\left(\begin{array}{ll}
1 & 0 \\
0 & e^{2 i k(x-y)}
\end{array}\right)\left(\begin{array}{c}
\frac{1}{2}A(y) f_{2} \\
- i(x-y) A^{*} f_{1}
\end{array}\right) d y.
$$
Similarly, $\Re_0^{\prime}$ is a bounded liner operator on $x\in(-\infty, 0]$ with
$$
\left\|\Re_0^{\prime}\right\|_{L^{\infty}(-\infty, 0]} \leq \left\|A\right\|_{L^{1,1}(\mathbb{R})}.
$$
So there is
$$
\begin{aligned}
\Psi^{-}_{1,k}(x, k)& \leq \sum_{n=1}^{\infty}\left|\partial_{k}\Re_0^{n}\left(\begin{array}{l}
1 \\
0
\end{array}\right)\right|\leq \sum_{n=0}^{\infty} \frac{1}{2}\frac{\left\|A\right\|_{L^{1}(\mathbb{R})^{n-1}}}{(n-1) !} \left\|A\right\|_{L^{1,1}(\mathbb{R})} \\
& \leq \left\|A\right\|_{L^{1,1}(\mathbb{R})} e^{\frac{1}{2}\left\|A\right\|_{L^{1}(\mathbb{R})}}
\end{aligned}
$$
for  $k \in \mathbb{R}$  and $x\in(-\infty, 0]$.

In similar steps, we can also get
$$
\Psi^{-}_{2,k}(x, k)\leq \left\|A\right\|_{L^{1,1}(\mathbb{R})} e^{\frac{1}{2}\left\|A\right\|_{L^{1}(\mathbb{R})}};
$$
$$
\Psi^{+}_{1,k}(x, k)\leq \left\|A\right\|_{L^{1,1}(\mathbb{R})} e^{\frac{1}{2}\left\|A\right\|_{L^{1}(\mathbb{R})}};
$$
$$
\Psi^{+}_{2,k}(x, k)\leq \left\|A\right\|_{L^{1,1}(\mathbb{R})} e^{\frac{1}{2}\left\|A\right\|_{L^{1}(\mathbb{R})}}.
$$
Next, we will estimate the characteristic function $(\Psi^{\pm})_{ij}(x, k),(i,j=1,2)$ respectively.  In fact, there are the following lemma
\begin{lemma}\label{l3}
For $k \in  \mathbb{R}$, we have the following estimates
$$
\begin{aligned}
&\left\|\psi_{21}^{-}(x,0,k)\right\|_{C^{0}\left(\mathbb{R}^{-}, L^{2}\left(I_{0}\right)\right)} \lesssim\|A\|_{H^{1,1}}; \\
&\left\|\psi_{21}^{-}(x,0,k)\right\|_{L^{2}\left(\mathbb{R}^{-} \times \mathbb{R}\right)} \lesssim\|A\|_{H^{1,1}};\\
&\left\|\psi_{12}^{-}(x,0,k)\right\|_{C^{0}\left(\mathbb{R}^{-}, L^{2}\left(\mathbb{R}\right)\right)} \lesssim\|A\|_{H^{1,1}}; \\
&\left\|\psi_{12}^{-}(x,0,k)\right\|_{L^{2}\left(\mathbb{R}^{-} \times \mathbb{R}\right)} \lesssim\|A\|_{H^{1,1}};\\
&\left\|\psi_{12}^{+}(x,0,k)\right\|_{C^{0}\left(\mathbb{R}^{+}, L^{2}\left(\mathbb{R}\right)\right)} \lesssim\|A\|_{H^{1,1}}; \\
&\left\|\psi_{12}^{+}(x,0,k)\right\|_{L^{2}\left(\mathbb{R}^{+} \times \mathbb{R}\right)} \lesssim\|A\|_{H^{1,1}};\\
&\left\|\psi_{21}^{+}(x,0,k)\right\|_{C^{0}\left(\mathbb{R}^{+}, L^{2}\left(\mathbb{R}\right)\right)} \lesssim\|A\|_{H^{1,1}}; \\
&\left\|\psi_{21}^{+}(x,0,k)\right\|_{L^{2}\left(\mathbb{R}^{+} \times \mathbb{R}\right)} \lesssim\|A\|_{H^{1,1}}.
\end{aligned}
$$
\end{lemma}
\begin{proof}
From the above integral form, we can know that
$$
\psi_{21}^{-}(x,0,k)=\int_{-\infty}^{x}-\frac{1}{2}A^{*}(y,k)\psi_{11}^{-}e^{2ik(x-y)}dy,
$$
then for $k \in  \mathbb{R}$, $\forall \sigma \in C_{0}\left(\mathbb{R}\right)$, we compute
$$
\begin{aligned}
\left\| \int_{-\infty}^{x}-\frac{1}{2}A^{*}\psi_{11}^{-}e^{2ik(x-y)}dy\right\|_{L^{2}\left(I_{0}\right)}^{2}&=s u p{ }_{\|\phi\|=1} \int_{I_{0}} \sigma(k) \int_{-\infty}^{x} -\frac{1}{2}A^{*}\psi_{11}^{-}e^{2ik(x-y)}dy\\
 &\lesssim ce^{\frac{1}{2}\left\|A\right\|_{L^{1}(\mathbb{R})}}\left\|A\right\|_{L^{1}},
 \end{aligned}
$$
this estimate is direct. In addition,
$$
\int_{\mathbb{R}^{+}} \int_{\mathbb{R}}\left|\int_{-\infty}^{x}-\frac{1}{2}A^{*}\psi_{11}^{-}e^{2ik(x-y)}dy\right|^{2} d k d x \lesssim \int_{\mathbb{R}^{+}}\left|\int_{-\infty}^{x}-\frac{1}{2}A^{*}\psi_{11}^{-}e^{2ik(x-y)}dy\right|^{2} d x \lesssim\|A\|_{L^{1,1}},
$$
which  implies that for $k \in \mathbb{R}$
$$
\begin{aligned}
&\left\|\psi_{21}^{-}(x,0,k)\right\|_{C^{0}\left(\mathbb{R}^{-}, L^{2}\left(\mathbb{R}\right)\right)} \lesssim\|A\|_{H^{1,1}}, \\
&\left\|\psi_{21}^{-}(x,0,k)\right\|_{L^{2}\left(\mathbb{R}^{-} \times \mathbb{R}\right)} \lesssim\|A\|_{H^{1,1}}.
\end{aligned}
$$
Similarly, others can also be obtained.
\end{proof}

Its derivative form for $k$ has
$$
\psi_{21,k}^{-}(x,0,k)=\int_{-\infty}^{x}-\frac{1}{2}A^{*}2i(x-y)\psi_{11}^{-}e^{2ik(x-y)}dy+\int_{-\infty}^{x}-\frac{1}{2}A^{*}\psi_{11,k}^{-}e^{2ik(x-y)}dy,
$$
Similarly, the following estimates can be obtained
$$
\begin{aligned}
&\left\|\psi_{21,k}^{-}(x,0,k)\right\|_{C^{0}\left(\mathbb{R}^{-}, L^{2}\left(\mathbb{R}\right)\right)} \lesssim\|A\|_{H^{1,1}} \\
&\left\|\psi_{21,k}^{-}(x,0,k)\right\|_{L^{2}\left(\mathbb{R}^{-} \times \mathbb{R}\right)} \lesssim\|A\|_{H^{1,1}}
\end{aligned}
$$
According to Eq.(\ref{s112})
$$
\begin{aligned}
&s_{11}(k)-1=(\psi^{-}_{11}-1)(\psi^{+}_{22}-1)+(\psi^{-}_{11}-1)+(\psi^{+}_{22}-1)-\psi^{-}_{21}\psi^{+}_{12}.\\
&s_{12}(k)=e^{2ikx}(\psi^{-}_{12}\psi^{+}_{22}-\psi^{-}_{22}\psi^{+}_{12}).
\end{aligned}
$$
Let's examine $\psi^{-}_{11}-1$ and $\psi^{+}_{22}-1$.
$$
\begin{aligned}
&\psi^{-}_{11}-1=\int_{-\infty}^{x}\frac{1}{2}A(y)\psi_{21}^{-}dy,\\
&\psi^{+}_{22}-1=\int_{-\infty}^{x}-\frac{1}{2}A^{*}(y)\psi_{12}^{+}dy,
\end{aligned}
$$
it's easy to know that $\psi^{-}_{11}-1$  is bounded and can be obtained by using Lemma \ref{l3} that
$$
\begin{aligned}
&\left\|\psi_{11}^{-}(x,0,k)-1\right\|_{C^{0}\left(\mathbb{R}^{-}, L^{2}\left(\mathbb{R}\right)\right)} \lesssim\|A\|_{H^{1,1}} \\
&\left\|\psi_{11}^{-}(x,0,k)-1\right\|_{L^{2}\left(\mathbb{R}^{-} \times \mathbb{R}\right)} \lesssim\|A\|_{H^{1,1}}.
\end{aligned}
$$
The estimation of $\psi^{+}_{22}-1$ can be obtained similarly.
For
$$
s_{11,k}=\psi^{-}_{11,k}\psi^{+}_{22}+\psi^{-}_{11}\psi^{+}_{22,k}-\psi^{-}_{21,k}\psi^{+}_{12}-\psi^{-}_{21}\psi^{+}_{12,k}
$$
where
$$
\psi^{-}_{11,k}=\int_{-\infty}^{x}\frac{1}{2}A(y,k)\psi_{21,k}^{-}dy, ~~~\psi^{+}_{22,k}=\int_{\infty}^{x}\frac{1}{2}A(y,k)\psi_{12,k}^{+}dy.
$$
A similar operation is performed on $s_{12}(k)$. Tracing back to the above estimates, when $k \in \mathbb{R}$ and  initial value $A \in H^{1,1}$, there is
$$
\begin{aligned}
s_{11}(k)\in L^{2}(\mathbb{R}),~~~s_{11,k}(k)\in L^{2}(\mathbb{R}),\\
s_{12}(k)\in L^{2}(\mathbb{R}),~~~s_{12,k}(k)\in L^{2}(\mathbb{R}).\\
\end{aligned}
$$

\section*{Appendix B: Solvable parabolic cylinder model}

Here we mainly describe the solution of the parabolic cylinder model introduced above. For the cAB equation studied in this paper, since there are two steady-state phase points, we need two parabolic cylinders to describe it. but their expansion forms are the same, using the following model.

For $r_0\in \mathbb{R}$, if $\nu=-\frac{1}{2 \pi} \log \left(1+\left|r_{0}\right|^{2}\right)$, define the contour $\Sigma^{PC}=\bigcup_{j=1}^{4} \Sigma_{j}$,
$$
\Sigma_{j}=\{\xi \in \mathbb{C}: \arg \xi=(2 j-1) \pi / 4\}, \quad j=1,2,3,4
$$
These four contours divide the plane into six areas $D_j,j=1..6$, as shown in Fig. \ref{16t}.
\begin{figure}[htpb]
{\includegraphics[width=9cm,height=4.5cm]{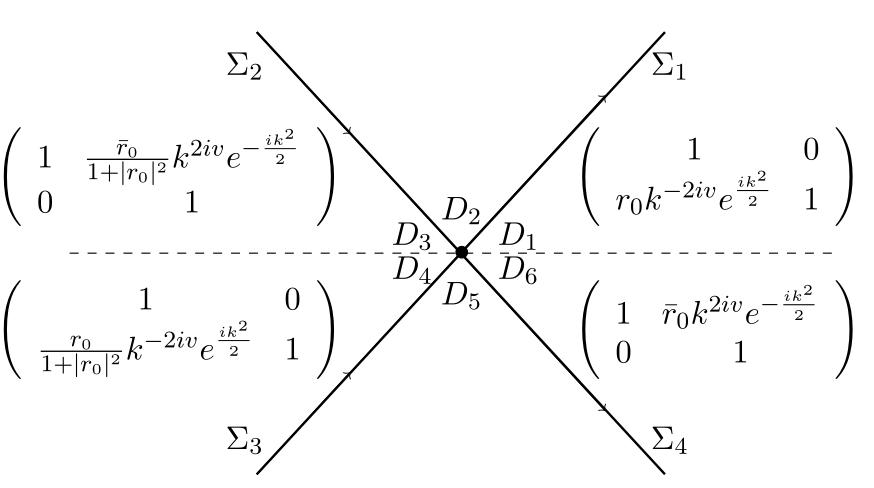}}
\centering
\caption{\small  The jump contour $V^{PC}$ for the $M^{PC}(k)$.}\label{16t}
\end{figure}

Therefore, we can consider the RHP corresponding to the following parabolic cylinder model.
\begin{rhp}\label{14r}
Find a matrix-valued function $M^{PC}(x,t,k)$ which satisfies:

$\bullet$  Analyticity: $M(x,t,k)$ is analytic in $\mathbb{C}\backslash\Sigma^{PC}$;

$\bullet$   Asymptotic behaviors: $M^{PC}(k)=I+\frac{M_{1}^{PC}}{ik}+O\left(k^{-2}\right), \quad k \rightarrow \infty;$

$\bullet$  Jump condition: $M^{PC}(x,t,k)$ has continuous boundary values $M^{PC}_{\pm}(x,t,k)$ on $\mathbb{R}$ and
\begin{equation}
M^{PC}_{+}(x, t, k)=M^{PC}_{-}(x, t, k) V^{PC}(k), \quad k \in \mathbb{R},
\end{equation}
where
\begin{equation}
V^{PC}(k)= \begin{cases}\left(\begin{array}{cc}
1 & 0 \\
r_{0} k^{-2 i \nu} e^{ \frac{ik^{2}}{2}} & 1
\end{array}\right), & k \in \Sigma_{1}, \\
\left(\begin{array}{cc}
1 & \frac{\bar{r}_{0}}{1+\left|r_{0}\right|^{2}} k^{2 i \nu} e^{- \frac{ik^{2}}{2}} \\
0 & 1
\end{array}\right), & k \in \Sigma_{2}, \\
\left(\begin{array}{cc}
1 & 0 \\
\frac{r_{0}}{1+\left|r_{0}\right|^{2}}k^{-2 i \nu} e^{ \frac{ik^{2}}{2}} & 1
\end{array}\right), & k \in \Sigma_{3}, \\
\left(\begin{array}{cc}
1& \bar{r}_{0} k^{2 i \nu} e^{- \frac{ik^{2}}{2}} \\
0&1
\end{array}\right), & k \in \Sigma_{4}.\end{cases}
\end{equation}
\end{rhp}
Make changes
\begin{equation}\label{pc}
M^{PC}=\vartheta \mathcal{P} k^{-i \nu \sigma_{3}} e^{\frac{ik^{2}}{4}  \sigma_{3}},
\end{equation}
where
$$
\mathcal{P}(k)= \begin{cases}\left(\begin{array}{cc}
1 & 0 \\
-r_{0} & 1
\end{array}\right), & k \in D_{1}, \\
\left(\begin{array}{cc}
1 & \frac{-\bar{r}_{0}}{1+\left|r_{0}\right|^{2}} \\
0 & 1
\end{array}\right), & k \in D_{3}, \\
\left(\begin{array}{cc}
1 & 0 \\
\frac{r_{0}}{1+\left|r_{0}\right|^{2}} & 1
\end{array}\right), & k \in D_{4}, \\
\left(\begin{array}{cc}
1 & \bar{r}_{0} \\
0 & 1
\end{array}\right), & k \in D_{6}, \\
\left(\begin{array}{ll}
1 & 0 \\
0 & 1
\end{array}\right), & k \in D_{2} \cup D_{5}.\end{cases}
$$
Then we can get a standard RHP with jump only in $k=0$.
\begin{rhp}\label{15r}
Find a matrix-valued function $\vartheta(x,t,k)$ which satisfies:

$\bullet$  Analyticity: $\vartheta(x,t,k)$ is analytic in $\mathbb{C}\backslash \mathbb{R}$;

$\bullet$   Asymptotic behaviors: $\vartheta(k) e^{\frac{i k^{2}}{4} \sigma_{3}} k^{-i \nu \sigma_{3}} =I+O\left(k^{-1}\right), \quad k \rightarrow \infty;$

$\bullet$  Jump condition: $ \vartheta_{+}(x,t,k)=\vartheta_{-}(x,t,k)v(0)$.
\end{rhp}
The above RH problem can be reduced to Weber equation,
$$
\partial_{z}^{2} D_{a}(z)+\left[\frac{1}{2}-\frac{z^{2}}{4}+a\right] D_{a}(z)=0,
$$
where $$
\begin{aligned}
&z=e^{-\frac{3 i \pi}{4}} k,~~~
a=i \Xi_{12} \Xi_{21}=i \nu,\\
&\Xi_{12}=\frac{\sqrt{2 \pi} e^{i \pi / 4 e^{-\pi \nu / 2}}}{r_{0} \Gamma(-a)}, \quad \Xi_{21}=\frac{-\sqrt{2 \pi} e^{-i \pi / 4 e^{-\pi \nu / 2}}}{\bar{r}_{0} \Gamma(a)}=\frac{\nu}{\Xi_{12}}.
\end{aligned}
$$
The solution of parabolic cylindrical explicit solution $\vartheta(k)$ can be given by parabolic cylindrical function. Obtained from Eq.(\ref{pc})
$$
M^{PC}(k)=I+\frac{M_{1}^{PC}}{ik}+O\left( k^{-2}\right),
$$
where
$$
M_{1}^{PC}=\left(\begin{array}{cc}
0 & \Xi_{12} \\
-\Xi_{21} & 0
\end{array}\right).
$$

\end{document}